\documentclass[a4paper,12pt]{amsart}
\usepackage{amsmath,amssymb,amsxtra,comment,graphicx,psfrag}
\usepackage{bm,mathrsfs}
\usepackage{mathtools}
\usepackage{xspace}
\usepackage{color}
\usepackage{xcolor}
\usepackage{hhline}

\definecolor{forestgreen}{rgb}{0.13, 0.55, 0.13}
\definecolor{lightblue}{rgb}{0.68, 0.85, 0.9}
\usepackage[colorlinks=true,
            linkcolor=blue,
            urlcolor  = blue,
            citecolor=forestgreen,
            anchorcolor = lightblue]{hyperref}

  \def\MR#1{\href{http://www.ams.org/mathscinet-getitem?mr=#1}{MR#1}}

\usepackage{enumerate}

\usepackage{hyperref}
\usepackage{pdfsync}
\usepackage{hhline}
\usepackage{fancyhdr}
\usepackage{epsfig}
\usepackage{epsfig,subfigure,epstopdf}
\allowdisplaybreaks

\usepackage{pstricks,pst-plot,pst-func}
\usepackage{pspicture}
\usepackage{curves}

\include{rgb}

\def\e{{\rm e}}
\def\ii{{\rm i}}
\def\dd{{\rm d}}
\def\C{{\mathbb C}}
\def\N{{\mathbb N}}
\def\P {{\mathbb P}}
\def\Re {{\mathbb R}}

\def\K{{\mathscr{K}}}
\def\D{{\mathscr{D}}}
\def\R{{\mathbb R}}

\DeclareMathOperator\Real {Re}

\definecolor{mygray}{rgb}{0.9,0.9,0.9}

\def\nst2{\| _*} 
 
\def\a12{A_h ^{1/2} } 

\def\tr|{|\! |\! |}

\def\e{\text {\rm e} }
\def\Co{{\mathbb C}}
\def\R {{\mathbb R}}
\def\P{{\mathbb P}}
\def\N{{\mathbb N}}

\def\N{{\mathbb N}}

\def\A{{\mathcal{A}}}
\def\B{{\mathcal{B}}}
\def\C{{\mathcal{C}}}
\def\D{{\mathcal{D}}}
\def\M{{\mathcal{M}}}

\def\K{{\mathscr{K}}}

\def \a{\alpha }

\def\T_h{{{\mathcal T}_h}}

\def\<{{\langle }}
\def\>{{\rangle }}

\begin{document}
\theoremstyle{plain}% default
\newtheorem{theorem}{Theorem}[section]
\newtheorem{lemma}{Lemma}[section]
\newtheorem{proposition}{Proposition}[section]
\newtheorem{corollary}{Corollary}[section]

\theoremstyle{definition}
\newtheorem{definition}[corollary]{Definition}

\newtheorem{example}{Example}[section]

\newtheorem{remark}{Remark}[section]
\newtheorem{remarks}[remark]{Remarks}
\newtheorem{note}{Note}
\newtheorem{case}{Case}

\numberwithin{equation}{section}
\numberwithin{table}{section}
\numberwithin{figure}{section}

\title[BDF methods for elliptic--parabolic systems]
{Implicit and implicit--explicit high-order BDF methods\\ for  coupled elliptic--parabolic systems}
\author[G. Akrivis]{Georgios Akrivis}
\address{Department of Computer Science and Engineering, University of Ioannina, 451$\,$10 Ioannina, Greece,
and Institute of Applied and Computational Mathematics, FORTH, 700$\,$13 Heraklion, Crete, Greece}
\email {\href{mailto:akrivis@cse.uoi.gr}{akrivis{\it @\,}cse.uoi.gr}}

\author[M. H. Chen]{Minghua Chen}
\address{School of Mathematics and Statistics, Gansu Key Laboratory of Applied Mathematics and Complex Systems,
Lanzhou University, Lanzhou 730000, P.R.\ China}
\email {\href{mailto:chenmh@lzu.edu.cn}{chenmh{\it @\,}lzu.edu.cn}}

\author[F. Yu]{Fan Yu}
\address{Institute for Math \& AI, Wuhan University, Wuhan 430072,  P.R.\ China}
\email {\href{mailto:yufan24@whu.edu.cn}{yufan24{\it @\,}whu.edu.cn}}

%\thanks{Work supported by NSFC 12471381 and the Science Fund for Distinguished Young Scholars of Gansu Province under Grant No.\ 23JRRA1020. Fan Yu is supported by NSFC 12501563 and the Postdoctor Project of Hubei Province under Grant No.\ 2024HBBHJD064.}

%\date{\today}
%
%\keywords{elliptic-parabolic system, 
\keywords{elliptic-parabolic system,  BDF methods, multipliers, energy technique, error estimates}
\subjclass[2020]{Primary 65M12, 65M60; Secondary 65L06.}

\begin{abstract}
First-order fully implicit as well as implicit--explicit schemes for coupled elliptic-parabolic systems are discussed 
in [Ern and Meunier, ESAIM: M2AN, 2009] and [Altmann et al., Math.\ Comp., 2021], respectively.
The extension of the analysis to higher-order (third-, fourth-, fifth-, and sixth-order) schemes 
is not straightforward since explicitly constructing $G$ matrices (G-stability) is often tricky.
In this article, we develop fully implicit as well as implicit--explicit  backward difference formula (BDF) schemes
of order up to six. The implicit--explicit variants are decoupled, thereby enhancing computational 
efficiency; their convergence analysis requires a weak coupling condition on the poroelastic parameters. 
In contrast, no coupling conditions  are needed for the fully implicit, coupled  schemes.
We determine novel and suitable multipliers for the two proposed classes and establish
error estimates via the energy technique. 
%High-order schemes play an important role in discretizing coupled elliptic-parabolic systems. 
A prominent advantage of these higher-order schemes is that, with almost the  computational cost 
of first-order schemes, they greatly improve the accuracy.
\end{abstract}

%***************%%%%%%%%%%%%%%%%%%%%%%%%%%%%%%%%%%%%%
%%%%%%%%%%%%%%%%%%%%%%%%%%%%%%%%%%%%%%%%%%%%%%%%%%%%%
\maketitle

\maketitle

\section{Introduction}
Let $\varOmega\subseteq\R^d, d=2,3,$ be a bounded  domain with Lipschitz boundary $\partial\varOmega.$ 
The unknowns of the system are the displacement field $u:[0,T]\times\varOmega\rightarrow\R^d$ and the 
pressure $p:[0,T]\times\varOmega\rightarrow\R$. For a given time horizon $T>0$, the equations of linear poroelasticity, \cite{Biot:41,Sh:00}, read
\begin{equation}\label{system}
\begin{split}
-\nabla\cdot\sigma(u)+\eta\nabla p&=f \quad\text{in}~~ [0,T]\times\varOmega, \\
\partial_t\Big(\eta\nabla\cdot u+\frac{1}{M}p\Big)-\nabla\cdot(\kappa\nabla p)&=g  \quad\text{in}~~ (0,T]\times\varOmega,
\end{split}
\end{equation}
subject to an initial condition $p(0)=p^0.$ 
Here, with $\varepsilon(u):=\frac{1}{2}\big(\nabla u+\left(\nabla u\right)^{\top}\big)$ the symmetric gradient
 used in continuum mechanics, $\sigma$ denotes the stress tensor,
\begin{equation*}
\sigma(u)=2\mu \varepsilon(u)+\lambda\left(\nabla\cdot u\right)I,
\end{equation*}
with Lam\'{e} coefficients $\lambda$ and $\mu$, $\kappa$ the ratio of the permeability and the fluid viscosity,
$\eta$ the Biot--Willis fluid-solid coupling coefficient, $M$ the Biot modulus, and $I$ the $d\times d$ identity matrix; 
see \cite{Biot:41}. The right-hand sides $f$ and $g$ are the volumetric 
load and the fluid source, respectively, modeling an injection
or production process. Throughout this paper, we assume homogeneous Dirichlet boundary conditions, 
 $u=0$ and $p=0$ on $(0,T]\times\partial\varOmega$.

This article is devoted to the analysis of  implicit and implicit--explicit BDF schemes for the linear 
poroelasticity model \eqref{system}  %, \cite{Biot:41,Sh:00}, 
with extensions to general linear elliptic-parabolic systems. 
The poroelastic equations can be formulated as a coupled system consisting of an elliptic and a parabolic equation. 
Significant applications of this problem include biomechanics, in which the human brain and heart are modeled 
as poroelastic media with multiple fluid networks, \cite{ERT:23}. Additionally, poroelastic equations arise 
as a model problem in geomechanics.

For the temporal discretization of elliptic-parabolic systems, such as the poroelasticity equations, 
a widely adopted strategy involves decoupled approaches, wherein the elliptic and
parabolic equations can be solved sequentially. This decoupling framework replaces
large coupled systems by two smaller subsystems, thereby enhancing the computational efficiency. 
A first-order semi-explicit scheme is proposed in \cite{AMU:21}; it decouples the system under a weak coupling condition. 
Later on, this framework was extended in \cite{AMU:24}  to a second-order semi-explicit scheme  by constructing a delay equation.
Recently, in \cite{AMU:25}, a third-order semi-explicit scheme was analyzed, and weighting matrices $G$
for methods of order up to $3$ were explicitly constructed.
However, the extension to higher-order schemes is non-trivial.
In contrast, fully implicit schemes, such as the implicit Euler method, \cite{EM:09}, maintain strong coupling but avoid
coupling conditions. This approach establishes a robust theoretical framework for unconditional stability and error estimates.

In this paper,  we analyze  high-order BDF methods, both implicit and implicit--explicit,
for coupled elliptic-parabolic systems. 
%The equivalence between A-stability and G-stability plays a core role in the convergence analysis. 
We  prove the A-stability property of auxiliary schemes and infer, utilizing the  equivalence between A- and G-stability,
existence of suitable positive definite symmetric matrices $G$. 
%To the best of our knowledge, this is the first time that a rigorous convergence analysis of  high-order BDF methods for elliptic-parabolic 
%problems is established.

An outline of the paper is as follows: We introduce the abstract formulation of coupled elliptic-parabolic systems in Section \ref{Se:Poro}.
In Section \ref{Se:Tempo}, we focus on time discretization by high-order BDF methods. 
In Section \ref{Se:4}, we determine suitable uniform multipliers.
Section \ref{Se:5} concerns related delay equations.
In Sections \ref{Se:Conc} and \ref{Se:Conv}, we establish consistency and  derive error estimates  for the proposed schemes. 
We conclude in Section \ref{Se:numerics} with numerical results.

\section{Abstract formulation}\label{Se:Poro}
In this section, we introduce an abstract formulation of \eqref{system} as an elliptic-parabolic system. 
We shall use the Hilbert spaces
\begin{equation*}
\mathcal{V}:=\left(H_0^1(\varOmega)\right)^d,\quad \mathcal{H_V}:=\left(L^2(\varOmega)\right)^d,
\quad \mathcal{W}:=H_0^1(\varOmega),\quad \mathcal{H_W}:=L^2(\varOmega).
\end{equation*}
%
%resulting in Gelfand triples. 
With the dual spaces $\mathcal{V}', \mathcal{W}'$ of $\mathcal{V}, \mathcal{W},$ 
respectively, %denoted by $\mathcal{V}^{\star}$ and $\mathcal{Q}^{\star}$, 
$\left(\mathcal{V},\mathcal{H_V},\mathcal{V}'\right)$ as well as $\left(\mathcal{W},\mathcal{H_W},\mathcal{W}'\right)$ 
form Gelfand triples with dense embeddings.
Moreover, we define the bilinear forms
\begin{equation*}
\begin{alignedat}{2}
&a(u,v):=\int_{\varOmega}\sigma(u):\varepsilon(v)\,\dd x,\qquad &&b(p,w):=\int_{\varOmega}\kappa\nabla p\cdot \nabla w\,\dd x,\\
&c(p,w):=\int_{\varOmega}\frac{1}{M} \,p \,w \,\dd x,\qquad  &&d(u,w):=\int_{\varOmega}\eta\left(\nabla\cdot u\right) w\,\dd x
\end{alignedat}
\end{equation*}
with the classical double dot notation, i.e., for matrices $A,B\in\Re^{n, m}$ we have $A:B=\text{trace}\left(A^{\top}B\right)$. 
%and the symmetric gradient $\varepsilon(u):=\frac{1}{2}\big(\nabla u+\left(\nabla u\right)^{\top}\big)$ used in continuum mechanics. 
With this, a weak formulation of \eqref{system} is as follows: seek $u:[0,T]\rightarrow\mathcal{V}$ 
and $p:[0,T]\rightarrow\mathcal{W}$ such that
\begin{equation}\label{weak}
\begin{alignedat}{2}
a(u,v)-d(v,p)&=( f,v)\quad &&\forall v\in\mathcal{V},\\
d(u_t,w)+c(p_t,w)+b(p,w)&=( g,w)\quad &&\forall w\in\mathcal{W}.
\end{alignedat}
\end{equation}
%
%for all test functions $v\in\mathcal{V},q\in\mathcal{Q}$. 
Correspondingly, we assume that the right-hand sides are such that $f:[0,T]\rightarrow\mathcal{V'}$,  
$g:[0,T]\rightarrow\mathcal{W'}$,  and denote by $( \cdot,\cdot)$ the
respective duality pairings. We  emphasize that it suffices to prescribe
initial data $p^0$ for $p$, since the first equation in \eqref{weak} imposes a consistency condition between $p^0$ and $u^0$.

The symmetric bilinear forms $a:\mathcal{V}\times\mathcal{V}\rightarrow\R$, $b:\mathcal{W}\times\mathcal{W}\rightarrow\R$, 
$c:\mathcal{H_W}\times\mathcal{H_W}\rightarrow\R$ are coercive and bounded; e.g., there exist positive constants $c_a$ and $C_a$ such that
\begin{equation*}
a(u,u)\geqslant c_a\|u\|^2_{\mathcal{V}}, \quad a(u,v)\leqslant C_a\|u\|_{\mathcal{V}}\|v\|_{\mathcal{V}}
\quad\forall u,v\in\mathcal{V}.
\end{equation*}
%
%for all $u,v\in\mathcal{V}$. 
%Similarly, $b:\mathcal{W}\times\mathcal{W}\rightarrow\R$ is symmetric, coercive, and bounded in $\mathcal{W}$, i.e.,
%there exist positive constants $c_b$ and $C_b$ such that
%%
%\begin{equation*}
%b(p,p)\geqslant c_b\left\|p\right\|^2_{\mathcal{W}}, \quad b(p,w)\leqslant C_b\left\|p\right\|_{\mathcal{W}}\left\|w\right\|_{\mathcal{W}}
%\quad\forall p,w\in\mathcal{W}.
%\end{equation*}
%%
%%for all $p,q\in\mathcal{Q}$.
%The bilinear form $c:\mathcal{H_W}\times\mathcal{H_W}\rightarrow\R$ simply involves multiplication by a (positive) 
%constant and, hence, defines an inner product on the pivot
%space $\mathcal{H_W}$. Thus, there exist positive constants  $c_c$ and $C_c$ such that
%%
%\begin{equation*}
%c(p,p)\geqslant c_c\left\|p\right\|^2_{\mathcal{H_W}}, \quad c(p,w)\leqslant C_c\left\|p\right\|_{\mathcal{H_W}}\left\|w\right\|_{\mathcal{H_W}}
%\quad\forall p,w\in\mathcal{H_W}.
%\end{equation*}
%
%for all $p,q\in\mathcal{H_Q}$. 
For convenience, we introduce the $a$-, $b$-, and $c$-norms,
$\|\cdot\|_a:=a(\cdot,\cdot)^{1/2}, \|\cdot\|_b:=b(\cdot,\cdot)^{1/2},$ and $\|\cdot\|_c=c(\cdot,\cdot)^{1/2},$ 
satisfying
\begin{equation*}
\frac{1}{C_a}\|\cdot\|_a^2\leqslant\|\cdot\|_\mathcal{V}^2\leqslant\frac{1}{c_a}\|\cdot\|_a^2,
\ \
\frac{1}{C_b}\|\cdot\|_b^2\leqslant\|\cdot\|_\mathcal{W}^2\leqslant\frac{1}{c_b}\|\cdot\|_b^2,
\ \ \frac{1}{C_c}\|\cdot\|_c^2\leqslant\|\cdot\|_\mathcal{H_W}^2\leqslant\frac{1}{c_c}\|\cdot\|_c^2,
\end{equation*}
respectively. The bilinear form $d:\mathcal{V}\times\mathcal{H_W}\rightarrow\R$ models the coupling
and is continuous, i.e., there exists a positive constant $C_d$ such that
\begin{equation*}
d(u,p)\leqslant C_d\left\|u\right\|_{\mathcal{V}}\left\|p\right\|_{\mathcal{H_W}}
\quad\forall u\in\mathcal{V}\ \forall p\in\mathcal{H_W}.
\end{equation*}
%
%for all $u\in\mathcal{V}$ and $p\in\mathcal{H_Q}$.

The variational formulation \eqref{weak} may also be written in operator form in the dual spaces $\mathcal{V}'$ and $\mathcal{W}'$.
Let $\mathcal{A}:\mathcal{V}\rightarrow\mathcal{V'}$, $\mathcal{B}:\mathcal{W}\rightarrow\mathcal{W}'$, 
$\mathcal{C}:\mathcal{H_W}\rightarrow\mathcal{H_W}$, and $\mathcal{D}:\mathcal{V}\rightarrow\mathcal{H_W}$ 
denote the bounded operators corresponding to the bilinear forms $a$, $b$, $c$, and $d$, respectively. 
We denote by $\mathcal{D}^{\star}$ the dual operator of $\mathcal{D}$.
Then, \eqref{weak} leads to the equivalent formulation
\begin{equation}\label{operator}
\begin{split}
\mathcal{A}u(t)-\mathcal{D^{\star}}p(t)&=f(t) \quad\text{in}~~ \mathcal{V'}, \\
\mathcal{D}u_t(t)+\mathcal{C}p_t(t)+\mathcal{B}p(t)&=g(t)  \quad\text{in}~~\mathcal{W'}.
\end{split}
\end{equation}
Since $\mathcal{A}$ is invertible, we can eliminate the variable $u$ and get a parabolic equation,
\begin{equation}\label{parabolic}
\left(\mathcal{M}+\mathcal{C}\right)p_t+\mathcal{B}p=g-\mathcal{D}\mathcal{A}^{-1}f_t;
\end{equation}
here $\mathcal{M}:=\mathcal{D}\mathcal{A}^{-1}\mathcal{D^{\star}}$  is a self-adjoint and non-negative operator.

\begin{remark}[Elliptic-parabolic systems]\label{Re:ell-par}
{\upshape
System \eqref{weak} can also be used to model linear thermoelasticity, which
considers the displacement of a material due to temperature changes, \cite{Biot:41}.
More generally, \eqref{weak} is an elliptic-parabolic system;  the elliptic part
(modeled by $a$) and the parabolic part (modeled by $b$ and $c$) are coupled through the
bilinear form $d$. We emphasize that the forthcoming analysis does not depend on the
specific application; it only depends on the properties of the bilinear forms.}
\end{remark}
\section{Temporal discretization}\label{Se:Tempo}
Let  $(\alpha, \beta)$ and $(\alpha, \gamma)$ be the implicit and explicit $q$-step BDF methods, 
respectively, generated  by the polynomials $\alpha, \beta$ and $\gamma,$
\[\alpha (\zeta)=\sum_{j=1}^q \frac 1j \zeta^{q-j} (\zeta -1)^j=\sum\limits^q_{i=0}\alpha_i \zeta ^{i}, 
\ \ \beta (\zeta)=\zeta^q, 
\ \ \gamma (\zeta)=\zeta^q-(\zeta -1)^q=\sum\limits^{q-1}_{i=0} \gamma_i \zeta^i,\]
$q=1,\dotsc,6.$ It is well known that the  implicit BDF schemes $(\alpha, \beta)$  are  strongly $A(0)$-stable for $q=1,\dotsc,6$
but are not even zero-stable for $q\geqslant 7.$

Let $N\in \N, \tau:=T/N$ be the constant time step, and $t_n :=n \tau, n=0,\dotsc,N,$
be a uniform partition of the interval $[0,T].$ 
Since we consider $q$-step schemes, we assume that we are given $q$ starting approximations 
$p^0,\dotsc,p^{q-1}\in \mathcal{W}.$  
In view of the first differential equation in \eqref{operator}, we then let  the corresponding approximations $u^0,\dotsc,u^{q-1}\in \mathcal{V}$ 
for $u$ be defined by
\begin{equation}
\label{starting-approx-u}
\A u^i-\D^\star p^i = f(t_i),\quad i=0,\dotsc,q-1.
\end{equation}

For sequences $(v^n)_{n\geqslant -q},$ we denote by ${\dot v}^n$ the \emph{discrete time derivative}
associated to the $q$-step BDF method,
\begin{equation}
\label{discrete-deriv}
{\dot v}^n:=\frac 1\tau\sum\limits^q_{i=0}\alpha_iv^{n-q+i},\quad n=0,1,\dotsc,
\end{equation}
and by  ${\hat v}^n$ the  \emph{extrapolated value} at $t_n$ of the polynomial of degree at most $q-1$ 
interpolating $(t_{n-q+i},v^{n-q+i}), i=0,\dotsc,q-1,$ 
\begin{equation}
\label{extrapolation}
{\hat v}^n:=\sum\limits^{q-1}_{i=0}\gamma_iv^{n-q+i},\quad n=0,1,\dotsc.
\end{equation}

We recursively define sequences of approximations $u^m\in \mathcal{V}, p^m\in \mathcal{W}$ to 
the nodal values $u(t_m,\cdot)$ and $p(t_m,\cdot)$ of the solutions $u$ 
and $p$ for the elliptic-parabolic system \eqref{operator}.
We shall use the notation  $f^m:=f(t_m,\cdot)$ and $g^m:=g(t_m,\cdot).$

\subsection{Fully implicit schemes}
Here, we discretize system \eqref{weak} by the implicit $q$-step BDF scheme $(\alpha, \beta),$
\begin{equation}\label{implicit}
\left.
\begin{alignedat}{2}
a(u^n,v)-d(v,p^n)&=\left(f^n,v\right)\quad &&\forall v\in\mathcal{V}\\
d({\dot u}^n,w)+c({\dot p}^n,w)+ b(p^n,w)&=\left(g^n,w\right)&&\forall w\in\mathcal{W}
\end{alignedat}
\right\}, \quad n=q,\dotsc,N,
\end{equation}
%
%for all $v\in\mathcal{V},q\in\mathcal{Q}$. 
%with $\langle \cdot,\cdot\rangle$ the corresponding duality pairings, 
or, equivalently, 
\begin{equation*}
\left.
\begin{split}
\mathcal{A}u^n-\mathcal{D^{\star}}p^n&=f^n\\ %\quad\text{in}~~ \mathcal{V^{\star}}, \\
\mathcal{D}{\dot u}^n+\mathcal{C}{\dot p}^n+\mathcal{B}p^n&=g^n%  \quad\text{in}~~\mathcal{Q^{\star}}.
\end{split}
\right\}, \quad n=q,\dotsc,N.
\end{equation*}
Now, in analogy to \eqref{parabolic}, in view of \eqref{starting-approx-u}, the first equation yields 
$\A{\dot u}^n=\D^\star {\dot  p}^n + {\dot f}^n, n=q,\dotsc,N,$ and, with 
 $\mathcal{M}=\mathcal{D}\mathcal{A}^{-1}\mathcal{D^{\star}}$, we can eliminate ${\dot u}^n$ from the
  second equation, and obtain
\begin{equation}
\label{full-a}
(\mathcal{M}+\mathcal{C}){\dot p}^n+ \B p^n= g^n-\D\A^{-1}{\dot f}^n,\quad n=q,\dotsc,N,
\end{equation}
that is
\begin{equation}
\label{full}
\left(\mathcal{M}+\mathcal{C}\right) \sum\limits^q_{i=0}\alpha_ip^{n-q+i}
+\tau \B p^n=\tau g^n-\D\A^{-1}\sum\limits^q_{i=0}\alpha_if^{n-q+i}, \quad n=q,\dotsc,N.
\end{equation}
Notice that $p^{n-q},\dotsc,p^n$ and $f^{n-q},\dotsc,f^n$ enter into \eqref{full}; in particular,
\eqref{full} has the form of a $q$-step scheme.
\subsection{Implicit--explicit schemes}\label{subsection:3.2}
First, we successively define $p^{-1},p^{-2},\dotsc, p^{-q}$, for $n=q-1,q-2,\dotsc,0,$ 
from the equations
\begin{equation}
\label{phat-p}
 {\hat  p}^n = p^n,\quad n=q-1,q-2, \dotsc,0;
\end{equation}
since $\gamma_0=(-1)^{q-1}\ne 0,$  $p^{-1},p^{-2},\dotsc, p^{-q}$ are indeed well defined.

Secondly, replacing $p^n$ in the first equation in \eqref{implicit} by the extrapolated value ${\hat p}^n$,
we  discretize system \eqref{weak} by the implicit--explicit $q$-step  BDF scheme $(\alpha, \beta, \gamma),$
\begin{equation}\label{explicit}
\left.
\begin{alignedat}{2}
a(u^n,v)-d\left(v,{\hat p}^n\right)&=\left(f^n,v\right)\quad &&\forall v\in\mathcal{V}\\
d({\dot u}^n,w)+c({\dot p}^n,w)+ b(p^n,w)&=\left(g^n,w\right)\quad &&\forall w\in\mathcal{W}
\end{alignedat}
\right\}, \quad n=q,\dotsc,N,
\end{equation}
or, equivalently, 
\begin{equation}\label{semi}
\left.
\begin{split}
\mathcal{A}u^n-\mathcal{D^{\star}}{\hat p}^n&=f^n\\  %\quad\text{in}~~ \mathcal{V^{\star}}, \\
\mathcal{D}{\dot u}^n+\mathcal{C}{\dot p}^n+\mathcal{B}p^n&=g^n  %\quad\text{in}~~\mathcal{Q^{\star}}.
\end{split}
\right\}, \quad n=q,\dotsc,N.
\end{equation}
The main computational advantage of \eqref{semi} is that the system
is decoupled: we first compute $u^n$ from the first equation and subsequently 
$ p^n$ from the second equation.
%Hence, the two equations decouple and can be solved sequentially.

Now, in view of  \eqref{phat-p} and  \eqref{starting-approx-u}, the first equation in 
\eqref{semi} is valid also for $n=0,\dotsc,q-1;$ this is the motivation for the 
 introduction of  $p^{-1},p^{-2},\dotsc, p^{-q}.$ This fact enables us to 
eliminate the variable $u$ from the second equation in \eqref{semi}
and write it in the same form for all relevant $n.$ Indeed, 
$\mathcal{A}u^\ell-\mathcal{D^{\star}}{\hat p}^\ell=f^\ell,  ~\ell=0,\dotsc,N,$
yields
\[\mathcal{A}{\dot u}^n-\mathcal{D^{\star}}\sum\limits^{q-1}_{j=0}\gamma_j {\dot p}^{n-q+j}
={\dot f}^n, \quad n=q,\dotsc,N,\]
and we easily infer that the second equation in \eqref{semi} takes the form
\begin{equation}
\label{abg3-a}
\mathcal{C}{\dot p}^{n}+\M\sum\limits^{q-1}_{j=0}\gamma_j{\dot p}^{n-q+j}
+ \B p^n= g^n-\D\A^{-1}{\dot f}^{n},\quad  n=q,\dotsc,N,
\end{equation}
i.e.,
\begin{equation}
\label{abg3}
\mathcal{C} \sum\limits^q_{i=0}\alpha_ip^{n-q+i}
+\M\sum\limits^q_{i=0}\sum\limits^{q-1}_{j=0}\alpha_i\gamma_jp^{n-2q+i+j}
+\tau \B p^n=\tau g^n-\D\A^{-1}\sum\limits^q_{i=0}\alpha_if^{n-q+i},
\end{equation}
$n=q,\dotsc,N.$ Notice that $p^{n-2q},\dotsc,p^n$ and $f^{n-q},\dotsc,f^n$ enter into \eqref{abg3}; in particular,
\eqref{abg3} has the form of a $2q$-step scheme.
%Let us also mention that the introduction of  $p^{-1},p^{-2},\dotsc, p^{-q}$ enabled us to write the scheme in the form  \eqref{abg3}
%even for $n=q,\dotsc,2q-1.$

Motivated by the first and the second term on the left-hand side of \eqref{abg3}, respectively,
let us introduce the polynomials
\begin{equation}
\label{characteristic-poly}
\tilde \alpha(\zeta):= \alpha(\zeta)\beta(\zeta)=\sum\limits^q_{i=0}\alpha_i \zeta ^{q+i}=\sum\limits^{2q}_{i=0}\tilde \alpha_i \zeta ^{i},\quad \hat \alpha(\zeta):=\alpha(\zeta)\gamma(\zeta)=\sum\limits^{2q-1}_{i=0}\hat \alpha_i \zeta ^{i}
\end{equation}
with
$\hat \alpha_i=\sum\limits^i_{j=0}\alpha_j\gamma_{i-j},~ i=0,\dotsc,2q-1;$
here, we used the notation $\alpha_{q+1}=\dotsb=\alpha_{2q-1}=0$ and $\gamma_q=\dotsb=\gamma_{2q-1}=0.$
With this notation, we can write \eqref{abg3} as
\begin{equation}
\label{numerical}
\mathcal{C} \sum\limits^{2q}_{i=0}\tilde \alpha_ip^{n-2q+i}+\M\sum\limits^{2q-1}_{i=0}\hat \alpha_ip^{n-2q+i}+\tau \B p^n=
\tau g^{n} -\D\A^{-1}\sum\limits^{2q}_{i=0}\tilde \alpha_if^{n-2q+i},
\end{equation}
$ n=q,\dotsc,N.$

%\begin{remark}[Estimation of $p^{-1},p^{-2},\dotsc, p^{-q}$ by $p^0,p^1,\dotsc, p^{q-1}$]\label{Re:hatp-p}  
%It follows immediately from \eqref{phat-p} that  $p^{-1},p^{-2},\dotsc, p^{-q}$ are linear combinations
%of $p^0,p^1,\dotsc, p^{q-1}$. In particular, $p^{-1},p^{-2},\dotsc, p^{-q}$ can be estimated by $p^0,p^1,\dotsc, p^{q-1}$
%in any relevant norm $|\cdot|$,
%%
%\begin{equation}
%\label{phat-p-estimate}
%|p^\ell|\leqslant c\big (|p^0|+\dotsb + |p^{q-1}|\big ),\quad \ell=-q, \dotsc,-1.
%\end{equation}
%%
%\end{remark}

 To ensure stability of the decoupled schemes, we shall need conditions on 
  the \emph{coupling strength}  $\omega$ between the elliptic and the parabolic equation,
\begin{equation}\label{strength}
\omega:=\frac{C_d^2}{c_ac_c}=\frac{\eta^2M}{\mu+\lambda}.
\end{equation}
%
%as the coupling strength between the elliptic and the parabolic equation. 
The value of $\omega$ depends on the physical coefficients of the application and
plays a crucial role in the convergence analysis of implicit--explicit schemes. 

\subsection{Necessary stability conditions for the implicit--explicit methods  \eqref{semi}}\label{SSe:3.3}
It is easily seen that
\begin{equation}\label{eq:est-M}
\|\M v\|_{\mathcal{H_W}}\leqslant  \omega\|\mathcal{C} v\|_{\mathcal{H_W}} \quad \forall v \in \mathcal{H_W}
\end{equation}
with $\mathcal{M}=\mathcal{D}\mathcal{A}^{-1}\mathcal{D^{\star}}$ and $\omega$ the constant
in  \eqref{strength}.

Indeed, obviously, $\mathcal{A}^{-1}:\mathcal{V}'\rightarrow\mathcal{V}$ and $\mathcal{D^{\star}}:\mathcal{H_W}\rightarrow\mathcal{V}'$,
since $\mathcal{A}:\mathcal{V}\rightarrow\mathcal{V}'$ and $\mathcal{D}:\mathcal{V}\rightarrow\mathcal{H_W}$. 
Thus, $\mathcal{M}:\mathcal{H_W}\rightarrow\mathcal{H_W}$,
and we obtain
\begin{equation*}
\begin{split}
\|\M v\|_{\mathcal{H_W}}&=\|\mathcal{D}\mathcal{A}^{-1}\mathcal{D^{\star}} v\|_{\mathcal{H_W}} 
\leqslant C_d\|\mathcal{A}^{-1}\mathcal{D^{\star}} v\|_{\mathcal{V}}
\leqslant \frac{C_d}{c_a}\|\mathcal{D^{\star}} v\|_{\mathcal{V}'}
\leqslant \frac{C_d^2}{c_a}\|v\|_{\mathcal{H_W}}\\
&=c_c\omega\|v\|_{\mathcal{H_W}}\leqslant  \omega\|\mathcal{C} v\|_{\mathcal{H_W}} \quad \forall v \in \mathcal{H_W}.
\end{split}
\end{equation*}
%
%({\color{red}Include details for these estimates.})
%i.e., \eqref{eq:est-M}.

\begin{lemma}[Necessary stability conditions for the implicit--explicit methods  \eqref{semi}]\label{Le:nece-cond}  
Consider a general elliptic-parabolic system and let $\M$ be dominated by $\mathcal{C},$ i.e.,
let \eqref{eq:est-M} be satisfied.
%
%
%
%%
%\begin{equation}
%\label{nece-stab1}
%\|\M v\|_{\mathcal{H_W}} \leqslant \omega \|\C v\|_{\mathcal{H_W}} \quad \forall v \in \mathcal{H_W}.
%\end{equation}
%%
Then, a necessary  stability condition for the implicit--explicit $q$-step BDF method  \eqref{semi}
is
\begin{equation*}
\omega\leqslant \frac 1{2^q-1}, \quad q=1,\dotsc,6.
\end{equation*}
\end{lemma}

\begin{proof}
Motivated by  \eqref{numerical}, consider the $2q$-step scheme
\begin{equation}
\label{nece-stab3}
\mathcal{C} \sum\limits^{2q}_{i=0}\tilde \alpha_iv^{n-2q+i}+\M\sum\limits^{2q-1}_{i=0}\hat \alpha_iv^{n-2q+i}
+\tau \B v^n=0,\quad  n=q,\dotsc,N.
\end{equation}
We shall show that for $\M=\ell \mathcal{C}$ with $\ell > 1/(2^q-1),$ 
the scheme \eqref{nece-stab3} is, in general, not unconditionally stable.
First, \eqref{nece-stab3} takes in this case the form
\[\mathcal{C} \sum\limits^{2q}_{i=0}(\tilde \alpha_i+\ell \hat \alpha_i)v^{n-2q+i}+\tau \B v^n=0,\]
with $\hat \alpha_{2q}=0,$ that is
\begin{equation*}
 \sum\limits^{2q}_{i=0}(\tilde \alpha_i+\ell \hat \alpha_i){\tilde v}^{n-2q+i}+\tau \mathcal{C}^{-1/2}\B \mathcal{C}^{-1/2} {\tilde v}^n=0,
 \quad  n=q,\dotsc,N,
\end{equation*}
with ${\tilde v}^j:= \mathcal{C}^{1/2} v^j.$ Notice that the operator $\mathcal{C}^{-1/2}\B \mathcal{C}^{-1/2}$ is positive
definite and self-adjoint.

Consider now the function $\kappa,$
\[\kappa(\zeta,x):=\tilde \alpha(\zeta)+\ell \hat \alpha(\zeta)+x[\beta(\zeta)]^2.\]
Since $\kappa(\cdot,x)$ is a polynomial of even degree $2q$ with a positive leading coefficient, we have
\begin{equation}
\label{nece-stab5}
\lim_{\zeta\to -\infty}\kappa(\zeta,x)= \infty.
\end{equation}
Furthermore,
$\alpha(-1)=c_q(-1)^q,~ \beta(-1)=(-1)^q,$ and  $\gamma(-1)=-(-1)^q(2^q-1),$
with a positive constant $c_q,$ whence
\begin{equation*}
\kappa(-1,x)=\tilde \alpha(-1)+\ell \hat \alpha(-1)+x[\beta(-1)]^2=
c_q\big [1-\ell (2^q-1)]+x.
\end{equation*}
Since the first term on the right-hand side is negative, for sufficiently small $x>0,$ we have
$\kappa(-1,x)<0.$ From this property and \eqref{nece-stab5}, we infer that there exist
$\zeta^\star <-1$ and $x^\star >0$ such that
\begin{equation}
\label{nece-stab7}
\kappa(\zeta^\star,x^\star)=0.
\end{equation}

According to the von Neumann criterion, a necessary stability condition is that,
if $\lambda>0$ is an eigenvalue of $\mathcal{C}^{-1/2}\B \mathcal{C}^{-1/2}$, the solutions of
\begin{equation*}
 \sum\limits^{2q}_{i=0}(\tilde \alpha_i+\ell \hat \alpha_i){\tilde v}^{n-2q+i}+\tau \lambda {\tilde v}^n=0
\end{equation*}
are bounded. For $\tau \lambda =x^\star$ this is not the case, since, in view of \eqref{nece-stab7},
the root condition is not satisfied; therefore, the scheme is not unconditionally stable.
\end{proof}

\section{Uniform multipliers for implicit--explicit BDF methods}\label{Se:4}
Multipliers for the (implicit) three-, four-, and five-step BDF methods were introduced
by Nevanlinna and Odeh in 1981 (see \cite{NO}) to make the energy
technique applicable to the stability analysis of these methods for parabolic
equations; no multipliers are required for the A-stable one- and two-step BDF methods.
The multiplier technique was first applied to the stability analysis  for parabolic equations 
in \cite{LMV}.
%became widely known and popular after its first actual application
%to the stability analysis  for parabolic equations by  Lubich, Mansour, and Venkataraman
%in 2013; see \cite{LMV}.

This technique hinges on the celebrated equivalence of A- and G-stability for multistep methods
by Dahlquist; see \cite{D}.

\begin{lemma}[\cite{D}; see also \cite{BC} and
{\cite[Section V.6]{HW}}]\label{Le:Dahl}
Let $\alpha(\zeta)=\alpha_q\zeta^q+\dotsb+\alpha_0$ and
$\kappa(\zeta)=\kappa_q\zeta^q+\dotsb+\kappa_0$ be polynomials of degree $q$,
with real coefficients, that have no common divisor.
Let $(\cdot,\cdot)$ be a real inner product with associated norm $|\cdot|.$
If
\begin{equation}
\label{A}
\Real \frac {\alpha(\zeta)}{\kappa(\zeta)}>0\quad\text{for }\, |\zeta|>1,
\tag{A}
\end{equation}
then there exists a positive definite symmetric matrix $G=(g_{ij})\in \R^{q,q}$
and real $\delta_0,\dotsc,\delta_q$ such that for $v^0,\dotsc,v^{q}$ in the inner product space,
\begin{equation}
\label{G}
 \Big (\sum_{i=0}^q\alpha_iv^{i},\sum_{j=0}^q\kappa_jv^{j}\Big )=
\sum_{i,j=1}^qg_{ij}(v^{i},v^{j})
-\sum_{i,j=1}^qg_{ij}(v^{i-1},v^{j-1})
+\Big |\sum_{i=0}^q\delta_iv^{i}\Big |^2.    % \eqno{\qed} %\qedhere       %\eqno{\qed}
\tag{G}
\end{equation}
\end{lemma}

Properties \eqref{A} and \eqref{G} mean that the $q$-step method $(\alpha,\kappa)$ is
A-stable and  G-stable, respectively. 

\begin{definition}[Multipliers and Nevanlinna--Odeh multipliers]\label{De:mult}
Let $\alpha$ be %the generating polynomial of the $q$-step BDF method.
a real polynomial of the degree $q$ and $\beta(\zeta)=\zeta^q.$
Consider a $q$-tuple $(\mu_1,\dotsc,\mu_q)$ of real numbers such that
with $\mu(\zeta):=\zeta^q-\mu_1\zeta^{q-1}-\dotsb-\mu_q$ and  the given $\alpha$, 
the pair $(\alpha,\mu)$ satisfies the A-stability condition \eqref{A}, with $\kappa(\zeta)$
replaced by $\mu(\zeta),$ and, in addition, the polynomials $\alpha$ and $\mu$ have no common divisor.
Then, $(\mu_1,\dotsc,\mu_q)$ are \emph{Nevanlinna--Odeh multipliers} or \emph{multipliers}, respectively, 
for the $q$-step method $(\alpha,\beta),$  if they are such that
\begin{equation*}
|\mu_1|+\dotsb+|\mu_q| <1
\end{equation*}
or  satisfy the milder \emph{positivity} property
\begin{equation}
\label{pos-prop}
1-\mu_1\cos \varphi-\dotsb-\mu_q\cos (q\varphi) >0 \quad \forall \varphi \in \R,
\tag{P}
\end{equation}
respectively; see \cite{ACYZ:21} for \eqref{pos-prop}.
\end{definition}

For $\mu_i\geqslant 0,$ \eqref{pos-prop} is equivalent to $|\mu_1|+\dotsb+|\mu_q| <1;$ otherwise, for $q\geqslant 2,$
it is a weaker condition.

\begin{remark}[Chebyshev polynomials and trigonometric identities]\label{Re:Cheb}
Recall  the Chebyshev polynomials $T_\ell$ and $U_\ell$ of the first and the second kind,
respectively,  $T_\ell(x) = \cos(\ell \arccos x)$ and   
$\sin (\ell \arccos x )=\sqrt{1-x^2} U_{\ell-1}(x), x\in [-1,1].$ 
They yield  the trigonometric identities
%With $x:=\cos\varphi,$ we have the trigonometric identities
%
\begin{equation}
\label{trig-id}
\cos (\ell \varphi)=T_\ell(x),\ \  \ell\in \N_0,  \quad \text{and}\quad \sin (\ell \varphi)=U_{\ell-1}(x)\sin \varphi,\ \ \ell\in \N,
\quad x:=\cos\varphi;
\end{equation}
we shall use these identities for $\ell=2,\dotsc,6$ in the sequel.  The Chebyshev polynomials satisfy the relations
$T_0(x)=U_0(x)=1, T_1(x)=x, U_1(x)=2x, T_{n+1}(x)=2xT_n(x)-T_{n-1}(x),$ and $U_{n+1}(x)=2U_n(x)-U_{n-1}(x),n=1,\dotsc .$

Let us also mention that   the positivity property \eqref{pos-prop} can be rewritten in the form
\begin{equation}
\label{pos-prop-Cheb}
1-\mu_1T_1(x)-\dotsb-\mu_qT_q(x) >0 \quad \forall x \in [-1,1].
\end{equation}
\end{remark}

Our objective here is the determination of \emph{uniform} multipliers for the
implicit--explicit $q$-step BDF schemes, $q=4,5,6,$ with 
generating polynomial $\check{\alpha},$ 
\begin{equation*}
\check{\alpha}(\zeta)= \tilde \alpha(\zeta)+m \hat \alpha(\zeta), %\quad 0\leqslant m\leqslant\omega\leqslant\frac{1}{2^q-1};
\end{equation*}
see \eqref{characteristic-poly} and \eqref{numerical}, for all parameters $m$ in the range $[0,1/(2^q-1)]$.
Notice that, for $m=0,$  $\check{\alpha}$ reduces to
 the generating polynomial $\check{\alpha}(\zeta)= \tilde \alpha(\zeta)=\zeta^q\alpha(\zeta)$
of the corresponding fully implicit scheme.

%Specifically, when $m=0$, the generating polynomial $\hat{\alpha}$ of the fully implicit schemes is
%\begin{equation*}
%\begin{split}
%\hat{\alpha}(\zeta)= \tilde \alpha(\zeta)=\zeta^q\alpha(\zeta).
%\end{split}
%\end{equation*}

%\subsection{Uniform multipliers for the implicit--explicit  six-step BDF method}

\begin{proposition}[Uniform multipliers for the implicit--explicit  six-step BDF method]\label{pro:six}
For $0\leqslant m\leqslant 1/63$, the set of numbers
\begin{equation}
\label{mu}
\mu_1=1,\quad \mu_2=-0.9,\quad \mu_3=0.3,\quad \mu_4=\mu_5=\mu_6=0,
\end{equation}
is a uniform multiplier for the implicit--explicit  six-step BDF method.
\end{proposition}

\begin{proof}
The proof consists of two parts; we first prove the A-stability property
 \eqref{A} and subsequently the positivity property \eqref{pos-prop}.

\emph{A-stability property \eqref{A}}.
The corresponding polynomial $\mu$ is
\begin{equation}
\label{mu1}
\mu(\zeta)=\zeta^{12}-\zeta^{11}+0.9\zeta^{10}-0.3\zeta^9
=\zeta^9\tilde{\kappa}(\zeta)\quad\text{with}\quad \tilde{\kappa}(\zeta)=\zeta^3-\zeta^2+0.9\zeta-0.3.
\end{equation}
Hence, to show that the roots of $\mu$ are inside the unit disk, it suffices to show that this is the case for $\tilde{\kappa}$. Now,
$\tilde{\kappa}(0.3)=-93/1000<0$,  $\tilde{\kappa}(0.5)=1/40>0,$
and thus $\tilde{\kappa}$ has a real root $\zeta_1\in (0.3,0.5).$ Actually, this is the only real root of $ \tilde{\kappa},$
since $ \tilde{\kappa}$ is strictly increasing on the real axis,
$\tilde{\kappa}'(x)=3x^2-2x+0.9>0$.
%since it does not have real roots. 

Let  $\zeta_2,\zeta_3$ be the complex conjugate roots of $\tilde{\kappa}.$
According to Vieta's formulas,
\begin{equation*}
\zeta_1\zeta_2\zeta_3=\zeta_1|\zeta_2|^2=0.3,
\end{equation*}
which, in combination with $\zeta_1>0.3,$ implies $|\zeta_2|<1.$
Thus, $|\zeta_1|,|\zeta_2|,|\zeta_3|<1$. We infer that all roots of $\mu$ are inside the unit disk.

The generating polynomial $\check{\alpha},$ for $q=6,$ is
\begin{equation*}
\begin{split}
60 \check{\alpha}(\zeta)&=147\zeta^{12}-(360-882m)\zeta^{11}+(450-4365m)\zeta^{10}-(400-11040m)\zeta^9\\
&\quad+(225-18555m)\zeta^8-(72-22632m)\zeta^7+(10-20864m)\zeta^6\\
&\quad+14700m\zeta^5-7815m\zeta^4+3030m\zeta^3-807m\zeta^2+132m\zeta-10m.
\end{split}
\end{equation*}
For  $0\leqslant m\leqslant 1/(2^q-1),$ $\tilde \alpha+m \hat \alpha$ is a convex combination
of  $\tilde \alpha$ and  $\tilde \alpha+1/(2^q-1)\hat \alpha;$ hence, to prove the A-stability property
\begin{equation*}
\Real \frac {\check{\alpha}(\zeta)}{\mu(\zeta)}
=\Real \frac {\tilde \alpha(\zeta)+m \tilde \gamma(\zeta)}{\mu(\zeta)}>0\quad\text{for }\, |\zeta|>1,\quad 0\leqslant m\leqslant\frac{1}{2^q-1},
\end{equation*}
it suffices to consider the two extreme cases, $m=0$ and $m=1/63.$

\emph{Case I: $m=0$}. We have
\begin{equation*}
60 \check{\alpha}(\zeta)=60 \zeta^{6}\alpha(\zeta)=\zeta^{6}\left(
147\zeta^{6}-360\zeta^{5}+450\zeta^{4}-400\zeta^3+225\zeta^2-72\zeta+10\right)
\end{equation*}
and $\mu(\zeta)=\zeta^9\tilde{\kappa}(\zeta)=\zeta^6\delta(\zeta)$ with $\delta(\zeta)=\zeta^3\tilde{\kappa}(\zeta).$

First, $\alpha(0)=1/6.$ 
Furthermore, $60\alpha(\zeta)=(\zeta-1)\chi(\zeta)$ with
\begin{equation*}
\begin{split}
\chi(\zeta):={}&147\zeta^{5}-213\zeta^{4}+237\zeta^{3}-163\zeta^2+62\zeta-10\\
={}&\Big (147\zeta^2-66\zeta+\frac{387}{10}\Big) \tilde{\kappa}(\zeta)-\frac{1}{100}\nu (\zeta)
\end{split}
\end{equation*}
and $\nu (\zeta):=2080\zeta^2-737\zeta-161,$ and none of the roots of the quadratic polynomial
$\nu$ is a root of $\tilde{\kappa};$ consequently, the polynomials $\chi$ and $\tilde{\kappa}$ have no common divisor.
We then easily infer that the polynomials $\alpha$ and $\delta$ do not have common divisor.

Now, $\alpha/\delta$ is holomorphic outside the unit disk in the
complex plane and
\[\lim_{|z|\to \infty}\frac {\alpha(z)}{\delta (z)}=\alpha_6=\frac {147}{60}>0.\]
Therefore, according to the maximum principle for harmonic functions,
the A-stability property \eqref{A} is equivalent to
\begin{equation*}
\Real \frac {\alpha(\zeta)}{\delta(\zeta)}\geqslant 0 \quad \forall \zeta\in \K,
\end{equation*}
with $\K:=\{\zeta\in \Co : |\zeta|=1\},$
i.e., equivalent to
\begin{equation}
\label{Real1}
\Real \big [\alpha(\e^{\ii \varphi})\delta (\e^{-\ii \varphi})\big ]\geqslant 0 \quad \forall \varphi \in \R.
\end{equation}
%

%Now, $\alpha(1)=0$, and we easily see that $\alpha(\zeta)=\frac 1{60}(\zeta-1)\tilde \alpha(\zeta)$ with
%%
%%
%\begin{equation}
%\label{alpha2}
%\tilde \alpha(\zeta):=147\zeta^5-213\zeta^4+237\zeta^3-163\zeta^2+62 \zeta- 10.
%\end{equation}
%%
%Therefore, \eqref{Real1} takes the form
%%
%\begin{equation}
%\label{Real2}
%\Real \Big [(\e^{\i \varphi}-1)\tilde \alpha(\e^{\i \varphi})
%\e^{-\i 3\varphi}\Big (\e^{-\i \varphi}-\frac 12\Big )^2\Big (\e^{-\i \varphi}-\frac 49\Big )\Big ]\geqslant 0 \quad \forall \varphi \in \Re.
%\end{equation}
%%

The desired property \eqref{Real1} takes the form
\begin{equation}
\label{Real2}
\Real \big [60\alpha(\e^{\ii \varphi})
\e^{-\ii 3\varphi} \big (10\e^{-\ii 3\varphi}-10\e^{-\ii 2\varphi}+ 9 \e^{-\ii \varphi}-3\big )\big]
\geqslant 0 \quad \forall \varphi \in \R.
\end{equation}
Now, it is easily seen that
\begin{equation*}
\begin{split}
60\alpha(\e^{\ii \varphi})\e^{-\ii 3\varphi}
&=\big [157 \cos(3\varphi)-432 \cos(2\varphi) +675 \cos\varphi-400\big ]\\
&\quad+\ii \big [137 \sin(3\varphi)-288 \sin(2\varphi) +225 \sin\varphi\big ].
\end{split}
\end{equation*}
With $x:=\cos\varphi,$ %recalling the elementary trigonometric identities
%%
%\[\cos(2\varphi) =2x^2-1,\  \cos(3\varphi) =4x^3-3x,\  \sin(2\varphi) =2x\sin\varphi,\  \sin(3\varphi) =(4x^2-1)\sin\varphi,\]
%%
%\[\left\{
%\begin{alignedat}{2}
%&\cos(2\varphi) =2\cos^2\varphi-1, \quad  && \cos(3\varphi) =4\cos^3\varphi-3\cos\varphi,\\
%&\sin(2\varphi) =2\cos\varphi \sin\varphi\quad  && \sin(3\varphi) =(4\cos^2\varphi-1)\sin\varphi,
%\end{alignedat}
%\right.\]
%%
using  trigonometric identities of the form \eqref{trig-id}, we see that
\begin{equation}\label{Real3}
60\alpha(\e^{\ii \varphi})\e^{-\ii 3\varphi}
=4(1-x)(8+59x- 157x^2)+\ii 4 (137x^2-144x+22)\sin\varphi.
\end{equation}
%
%here and in the following we use the notation $x:=\cos\varphi.$
Notice that the factor $1-x$ in the real part of $\alpha(\e^{\ii \varphi})\e^{-\ii 3\varphi}$
is due to the fact that $\alpha(1)=0.$ Similarly,
\[\begin{split}
10\e^{-\ii 3\varphi}-10\e^{-\ii 2\varphi}+ 9 \e^{-\ii \varphi}-3
&=\big [10 \cos(3\varphi)-10 \cos(2\varphi) +9 \cos\varphi-3\big ]\\
&\quad-\ii \big [10 \sin(3\varphi)-10 \sin(2\varphi) +9 \sin\varphi\big ]
\end{split}\]
and
\begin{equation}
\label{Real4}
10\e^{-\ii 3\varphi}-10\e^{-\ii 2\varphi}+ 9 \e^{-\ii \varphi}-3
=40x^3 - 20x^2 - 21x + 7-\ii (40x^2 - 20x - 1)\sin\varphi.
\end{equation}
In view of \eqref{Real3} and \eqref{Real4}, the desired property  \eqref{Real2}
can be written in the form
\begin{equation*}
4(1-x)P(x)\geqslant 0 \quad \forall x \in [-1,1]
\end{equation*}
with
%
%\[\begin{split}
%P(x):=&(8+59x-157x^2)(40x^3 - 20x^2 - 21x + 7)\\
%&+(1+x)(137x^2-144x+22)(40x^2 - 20x - 1),
%\end{split}\]
%%
%i.e.,
%
\begin{equation*}
\begin{split}
P(x)&=-800x^5 + 2480x^4 - 2440x^3 + 829x^2  - 73x + 34\\
&=x^2(-800x^3 + 2480x^2 - 2440x+789)+40x^2  - 73x + 34.
\end{split}
\end{equation*}
Now, $P$ is positive in the interval $[-1,1],$ and thus   \eqref{Real2} is valid.
Indeed, first, the quadratic polynomial $40x^2-73x + 34$ is positive for all real $x$,
since it does not have real roots. 
Also, the cubic polynomial  $-800x^3 + 2480x^2 -2440x+789$ is positive in the interval $[-1,1]$. 
In fact, all terms are positive for negative $x$; for the positivity of this cubic polynomial in
$[0,1]$ see  Figure \ref{Fig:P} (left).

\emph{Case II: $m=1/63$}. We have $\mu(\zeta)=\zeta^9\tilde{\kappa}(\zeta)$ and
\begin{equation*}
\begin{split}
60 \check{\alpha}(\zeta)63&=9261\zeta^{12}-21798\zeta^{11}+23985\zeta^{10}-14160\zeta^9-4380\zeta^8+18096\zeta^7\\
&\quad-20234\zeta^6+14700\zeta^5-7815\zeta^4+3030\zeta^3-807\zeta^2+132\zeta-10.
\end{split}
\end{equation*}
%
%and $\mu(\zeta)=\zeta^9\tilde{\kappa}(\zeta).$

First, $\check{\alpha}(0)=-1/378.$ 
Furthermore, $60 \check{\alpha}(\zeta)63=(3\zeta-1)(\zeta^2-1)(3\zeta^2+1)(7\zeta^2-4\zeta+1)\chi(\zeta)$;
here, $\chi$ and $\tilde{\kappa}$ are the polynomials of Case I.
Since $\chi$ and $\tilde{\kappa}$ have no common divisor, we infer that 
the polynomials $\check{\alpha}$ and $\mu$ do not have common divisor.

Now, $\check{\alpha}/\mu$ is holomorphic outside the unit disk in the
complex plane and
\[\lim_{|z|\to \infty}\frac {\check{\alpha}(z)}{\mu (z)}=\frac {147}{60}>0.\]
As before, the A-stability property \eqref{A} is equivalent to
\begin{equation*}
\Real \frac {\check{\alpha}(\zeta)}{\mu(\zeta)}\geqslant 0 \quad \forall \zeta\in \K,
\end{equation*}
%
%with $\K$ the unit circle in the complex plane, $\K:=\{\zeta\in \Co : |\zeta|=1\},$
i.e., equivalent to
\begin{equation}
\label{RReal1}
\Real \big [\check{\alpha}(\e^{\ii \varphi})\mu (\e^{-\ii \varphi})\big ]\geqslant 0 \quad \forall \varphi \in \R.
\end{equation}
%

%Now, $\alpha(1)=0$, and we easily see that $\alpha(\zeta)=\frac 1{60}(\zeta-1)\tilde \alpha(\zeta)$ with
%%
%%
%\begin{equation}
%\label{alpha2}
%\tilde \alpha(\zeta):=147\zeta^5-213\zeta^4+237\zeta^3-163\zeta^2+62 \zeta- 10.
%\end{equation}
%%
%Therefore, \eqref{Real1} takes the form
%%
%\begin{equation}
%\label{Real2}
%\Real \Big [(\e^{\i \varphi}-1)\tilde \alpha(\e^{\i \varphi})
%\e^{-\i 3\varphi}\Big (\e^{-\i \varphi}-\frac 12\Big )^2\Big (\e^{-\i \varphi}-\frac 49\Big )\Big ]\geqslant 0 \quad \forall \varphi \in \Re.
%\end{equation}
%%

In view of \eqref{mu1}, the desired property \eqref{RReal1} takes the form
\begin{equation}
\label{RReal2}
\Real \big [60\check{\alpha}(\e^{\ii \varphi})
\e^{-\ii 6\varphi}63\cdot 10\mu(\e^{-\ii \varphi})
\e^{\ii 6\varphi}]\geqslant 0 \quad \forall \varphi \in \R.
\end{equation}
Now, it is easily seen that
\begin{equation*}
\begin{split}
60\check{\alpha}(\e^{\ii \varphi})\e^{-\ii 6\varphi}63
&=9251\cos(6\varphi)-21666\cos(5\varphi)+23178\cos(4\varphi)-11130\cos(3\varphi)\\
&\quad-12195\cos(2\varphi)+32796 \cos\varphi-20234 \\
&\quad+\ii\big [9271\sin(6\varphi)-21930 \sin(5\varphi)
+24792\sin(4\varphi)-17190\sin(3\varphi)\\
&\quad+3435\sin(2\varphi)+3396\sin\varphi\big ].
\end{split}
\end{equation*}
With $x:=\cos\varphi,$ using  trigonometric identities of the form \eqref{trig-id},
%recalling the elementary trigonometric identities
%%
%\begin{equation*}
%\begin{alignedat}{2}
%&\sin(4\varphi) =(8x^3-4x)\sin\varphi, \quad &&\cos(4\varphi) =8x^4-8x^2+1, \\
%&\sin(5\varphi) =(16x^4-12x^2+1)\sin\varphi, \quad  &&\cos(5\varphi) =16x^5-20x^3+5x,\\
%&\sin(6\varphi) =(32x^5-32x^3+6x)\sin\varphi, \quad  &&\cos(6\varphi) =32x^6-48x^4+18x^2-1,
%\end{alignedat}
%\end{equation*}
%
%\[\left\{
%\begin{alignedat}{2}
%&\cos(2\varphi) =2\cos^2\varphi-1, \quad  && \cos(3\varphi) =4\cos^3\varphi-3\cos\varphi,\\
%&\sin(2\varphi) =2\cos\varphi \sin\varphi\quad  && \sin(3\varphi) =(4\cos^2\varphi-1)\sin\varphi,
%\end{alignedat}
%\right.\]
%%
we  see that
\begin{equation}\label{RReal3}
\begin{split}
60\check{\alpha}(\e^{\ii \varphi})&\e^{-\ii 6\varphi}63
=32(1-x^2)\left(-9251x^4 + 10833x^3 - 1169x^2 - 1317x + 184 \right)\\
&\quad+\ii 32\left( 9271x^5 - 10965x^4 - 3073x^3 + 6075x^2 - 1146x - 42\right)\sin\varphi.
\end{split}
\end{equation}
Similarly,
\[\begin{split}
10\mu(\e^{-\ii \varphi})\e^{\ii 6\varphi}
&=10 \cos(6\varphi)-10 \cos(5\varphi)+9 \cos(4\varphi)-3 \cos(3\varphi)\\
&\quad-\ii \big [10 \sin(6\varphi)-10 \sin(5\varphi)+9 \sin(4\varphi)-3 \sin(3\varphi)\big ]
\end{split}\]
and
\begin{equation}
\label{RReal4}
\begin{split}
10\mu(\e^{-\ii \varphi})\e^{\ii 6\varphi}
&= 320x^6 - 160x^5 - 408x^4 + 188x^3 + 108x^2 - 41x - 1\\
&\quad-\ii (320x^5 - 160x^4 - 248x^3 + 108x^2 + 24x - 7)\sin\varphi.
\end{split}
\end{equation}
In view of \eqref{RReal3} and \eqref{RReal4}, the desired property  \eqref{RReal2}
can be written in the form
\begin{equation*}
32(1-x^2)P(x)\geqslant 0 \quad \forall x \in [-1,1]
\end{equation*}
with
%%
%\[\begin{split}
%P(x):=&\left(-9251x^4 + 10833x^3 - 1169x^2 - 1317x + 184 \right)\\
%&\times(320x^6 - 160x^5 - 408x^4 + 188x^3 + 108x^2 - 41x - 1)\\
%&+\left( 9271x^5 - 10965x^4 - 3073x^3 + 6075x^2 - 1146x - 42\right)\\
%&\times(320x^5 - 160x^4 - 248x^3 + 108x^2 + 24x - 7),
%\end{split}\]
%%
%i.e.,
%
\begin{equation}\label{PP2}
\begin{split}
P(x)&=6400x^{10} - 45440x^9 + 138880x^8 - 237184x^7 + 245716x^6 - 159242x^5\\
&\quad + 66209x^4 - 16589x^3 + 473x^2 + 787x + 110.
\end{split}
\end{equation}

Now, $P$ is positive in the interval $[-1,1]$;  thus, \eqref{RReal2} is valid. 
Indeed, first, the cubic polynomial  $-16589x^3 + 473x^2 + 787x + 110$ is positive for negative $x$; 
see  Figure \ref{Fig:P} (middle); all other terms are positive for negative $x$. 
Furthermore, for $0\leqslant x  \leqslant 1,$ $P$ is positive; see  Figure \ref{Fig:P} (right).

\begin{figure}[!ht]
\centering
\scalebox{0.8}{
% -800x^3 + 2480x^2 -2440x+789
%\framebox{
\psset{yunit=0.003cm,xunit=3.5cm}
\begin{pspicture}(-0.10,-86.1)(1.16,910)
\psaxes[ticks=none,labels=none,linewidth=0.6pt]{->}(0,0)(-0.13,-85.8)(1.12,852)%
[{\footnotesize $\!\!\!x$},0][$ $,90]
\psset{plotpoints=10000}
\psPolynomial[coeff=789 -2440 2480 -800, linewidth=0.5pt,linecolor=blue]{0}{1}
\uput[0](0.33,250){\small $p$}
\uput[0](-0.072,895){\footnotesize $y$}
\uput[0](0.93,-65){\footnotesize $1$}
\pscircle*(0,200){0.038}
\pscircle*(0,400){0.038}
\pscircle*(0,600){0.038}
\uput[0](-0.21,200){\scriptsize $200$}
\uput[0](-0.21,400){\scriptsize $400$}
\uput[0](-0.21,600){\scriptsize $600$}
\uput[0](-0.14,-65.6){\footnotesize $O$}
\pscircle*(1,0){0.038}
\end{pspicture} %}
\quad
%\framebox{
\psset{yunit=0.00015cm,xunit=3.5cm}
\begin{pspicture}(-1.035,-1700)(0.16,18010)
\psaxes[ticks=none,labels=none,linewidth=0.6pt]{->}(0,0)(-1.06,-1700)(0.12,16700)%
[{\footnotesize $\!\!\!x$},0][$ $,90]
\psset{plotpoints=10000}
\psPolynomial[coeff=110  787  473  -16589, linewidth=0.5pt,linecolor=blue]{-1}{0}
\uput[0](-0.68,5250){\small $p$}
\uput[0](-0.072,17570){\footnotesize $y$}
\uput[0](-1.125,-1350){\scriptsize $-1$}
\pscircle*(0,6000){0.038}
\pscircle*(0,12000){0.038}
\uput[0](-0.255,6000){\scriptsize $6000$}
\uput[0](-0.299,12000){\scriptsize $12000$}
\uput[0](-0.14,-1300){\footnotesize $O$}
\pscircle*(-1,0){0.038}
\end{pspicture} %}
\quad 
%\framebox{
\psset{yunit=0.01cm,xunit=3.5cm}
\begin{pspicture}(-0.10,-28.1)(1.16,273)
\psaxes[ticks=none,labels=none,linewidth=0.6pt]{->}(0,0)(-0.11,-25.8)(1.12,250)%
[{\footnotesize $\!\!\!x$},0][$ $,90]
\psset{plotpoints=10000}
\psPolynomial[coeff=110 787 473  -16589 66209  -159242  245716  -237184 138880 -45440 6400, linewidth=0.5pt,linecolor=blue]{0}{1}
\uput[0](0.53,150){\small $P$}
\uput[0](-0.072,265){\footnotesize $y$}
\uput[0](0.93,-21){\scriptsize $1$}
\pscircle*(0,100){0.038}
\pscircle*(0,200){0.038}
\uput[0](-0.21,100){\scriptsize $100$}
\uput[0](-0.21,200){\scriptsize $200$}
\uput[0](-0.14,-20){\footnotesize $O$}
\pscircle*(1,0){0.038}
\end{pspicture} %}
}
\caption{The graphs of  polynomials $p(x)=-800x^3 + 2480x^2 -2440x+789$
in the interval $[0,1]$ (left), $p(x)=-16589x^3 + 473x^2 +787x+110$ %of \eqref{Real6}
in $[-1,0]$ (middle), and  $P$  of \eqref{PP2} in  $[0,1]$ (right).}
\label{Fig:P}
\end{figure}
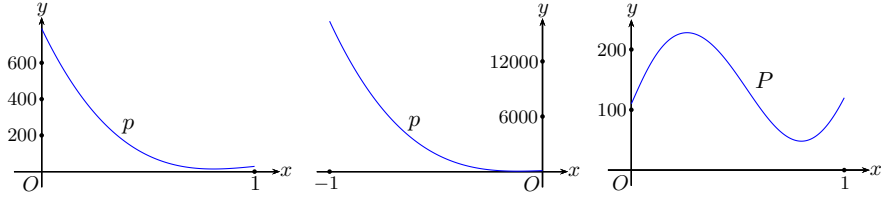

\emph{Positivity property \eqref{pos-prop}.}
Here, we prove the desired positivity property \eqref{pos-prop} for the multiplier \eqref{mu}.
Actually, since in the stability analysis we will use the value $\varepsilon=9/100$,
we shall directly prove that the function
in  \eqref{pos-prop} for the multiplier \eqref{mu} is bounded from below by $9/100.$ To this end,  we subtract
$9/100$ from the corresponding expression, and shall show that the function $g$,
\begin{equation}
\label{f}
g(s):=\frac{91}{100} -\cos s+\frac{9}{10}\cos(2s)-\frac{3}{10}\cos(3s), \quad s\in \R,
\end{equation}
is positive. Now, with $x:=\cos s,$  trigonometric identities of the form \eqref{trig-id}
yield $g(s)=p(x)$ with $p$  the polynomial 
\begin{equation*}
p(x):=-\frac{6}{5}x^3+\frac{9}{5}x^2-\frac{1}{10}x+\frac{1}{100},\quad x\in [-1,1].
\end{equation*}
%
%Now, elementary trigonometric identities lead to the following form of $g,$
%%
%\[g(s)=-\frac{6}{5}\cos^3s+\frac{9}{5}\cos^2s-\frac{1}{10}\cos s+\frac{1}{100}.\]
%%
%Hence, we consider the polynomial $p,$
%%
%\begin{equation*}
%p(x):=-\frac{6}{5}x^3+\frac{9}{5}x^2-\frac{1}{10}x+\frac{1}{100},\quad x\in [-1,1].
%\end{equation*}
%
It is easily seen that $p$ attains its minimum in $[-1,1]$ at $x^{\star}=(3-2\sqrt{2})/6$
and
\[p(x^\star)>0.008584 > 0.\]
Therefore, $g$ is indeed positive; 
the desired positivity property \eqref{pos-prop} is satisfied. 
\end{proof}

Analogously to Proposition \ref{pro:six}, we can establish the following results; we omit the proofs 
for the sake of brevity. 

\begin{proposition}[Uniform multipliers for the implicit--explicit  four-step BDF method]\label{pro:four}
For $0\leqslant m\leqslant 1/15$, the set of numbers
\begin{equation}\label{mu4}
\mu_1=0.5,\quad \mu_2=\mu_3=\mu_4=0,
\end{equation}
is a uniform multiplier for the implicit--explicit  four-step BDF method.  
\end{proposition}

\begin{proposition}[Uniform multipliers for the implicit--explicit  five-step BDF method]\label{pro:five}
For $0\leqslant m\leqslant 1/31$, the set of numbers
\begin{equation}\label{mu5}
\mu_1=1,\quad \mu_2=-0.25,\quad \mu_3=\mu_4=\mu_5=0,
\end{equation}
is a uniform multiplier for the implicit--explicit  five-step BDF method.
\end{proposition}

\section{Related delay equations}\label{Se:5}
Following an idea from \cite{AMU:21,AMU:24}, we first construct a  delay system for the coupled elliptic-parabolic system \eqref{weak}
and then discretize the delay  system by  implicit  BDF methods;
this results in implicit--explicit  BDF methods for the original elliptic-parabolic system \eqref{operator}.
%(This section is devoted for implicit--explicit schemes, the delay system \label{adeq1} has approximated $\mathfrak{p}$ explicitly, then discretize
% the delay system by implicit BDF methods)
In more detail, we consider the delay system
\begin{equation}\label{adeq1}
\begin{alignedat}{2}
a(\mathfrak{u},v)-d(v,\hat{\mathfrak{p}})
&=\left(f,v\right)\quad &&\forall v\in\mathcal{V},\\
d(\mathfrak{u}_t,w)+c(\mathfrak{p}_t,w)+b(\mathfrak{p},w)
&=\left(g,w\right) \quad &&\forall w\in\mathcal{W},
\end{alignedat}
\end{equation}
with $\hat{\mathfrak{p}}$ an approximation of $\mathfrak{p}$ in the interval $[0,T]$, depending on the specific BDF method
as well as on the time step $\tau,$  namely, in analogy to \eqref{extrapolation},
\begin{equation}
\label{extrapolation-delay}
\hat{\mathfrak{p}}(t):=\sum\limits^{q-1}_{i=0}\gamma_i\mathfrak{p}(t-q\tau+i\tau),
\quad t\in [0,T].
\end{equation}
Let us emphasize that due to \eqref{extrapolation-delay}, in contrast to
the original system \eqref{weak}, such a delay system calls for a history function $\mathfrak{p}|_{[-q\tau,0]}(t)=\varPhi(t)$ in $[-q\tau,0]$ rather
than only for an initial value, i.e., a prescription of $\mathfrak{p}$ for $-q\tau\leqslant t\leqslant0.$

Then, \eqref{adeq1} leads to the equivalent formulation
\begin{equation}\label{delay}
\begin{alignedat}{2}
\mathcal{A}\mathfrak{u}(t)-\mathcal{D^{\star}}\hat{\mathfrak{p}}(t)&=f(t) \quad &&\text{in}~~ \mathcal{V}', \\
\mathcal{D}\mathfrak{u}_t(t)+\mathcal{C}\mathfrak{p}_t(t)+\mathcal{B}\mathfrak{p}(t)&=g(t)  \quad &&\text{in}~~\mathcal{W}',
\end{alignedat}
\end{equation}
for $t\in (0,T).$ Substitution of the expression for $\hat{\mathfrak{p}}$ from \eqref{extrapolation-delay} into \eqref{delay} yields
\begin{equation*}
\begin{split}
\mathcal{A}\mathfrak{u}(t)-\mathcal{D^{\star}}
\sum\limits^{q-1}_{i=0}\gamma_i\mathfrak{p}(t-q\tau+i\tau)&=f(t) \quad\text{in}~~ \mathcal{V}', \\
\mathcal{D}\mathfrak{u}_t(t)+\mathcal{C}\mathfrak{p}_t(t)
+\mathcal{B}\mathfrak{p}(t)&=g(t)  \quad\text{in}~~\mathcal{W}'.
\end{split}
\end{equation*}
Elimination of the variable $\mathfrak{u}$, yields  the delay parabolic equation
\begin{equation}\label{adeq144}
\begin{split}
\mathcal{C}\mathfrak{p}_t(t)+\mathcal{M}\sum\limits^{q-1}_{j=0}\gamma_j\mathfrak{p}_t(t-q\tau+j\tau)
+\mathcal{B}\mathfrak{p}(t)
=g(t)-\mathcal{D}\mathcal{A}^{-1}f_t(t).
\end{split}
\end{equation}

In \eqref{adeq1}, we seek functions $\mathfrak{u}:[0,T]\rightarrow\mathcal{V}$
and $\mathfrak{p}:[-q\tau,T]\rightarrow\mathcal{W}$, given sufficiently smooth right-hand sides and a history function
$\varPhi\in C^{\infty}([-q\tau,0];\mathcal{H_W})$ for $\mathfrak{p}$ that ensures consistency of the initial conditions 
of the original problem
\eqref{weak}. 
For this, a sufficient condition reads (\cite{AMU:24})
\begin{equation}\label{adeq2}
\varPhi(t_{-\ell})=p^{-\ell},\quad \ell=1,\ldots,q, 
\quad \varPhi(0)=p(0)=p^0;
\end{equation}
see  \eqref{phat-p} and \eqref{start-acc} for the definition of $p^{-\ell}$ and for the requirement on $p^{\ell},$
respectively. By \eqref{adeq2} and the first equation in \eqref{adeq1}, we conclude that $\mathfrak{u}(0)=u(0)=u^0$,
since 
\begin{equation*}
a(\mathfrak{u}(0),v)=(f(0),v)+ d(v,\hat{\mathfrak{p}}(0))
=(f(0),v)+ d(v,p(0)) \quad \forall v\in \mathcal{V}.
\end{equation*}

\begin{proposition}[On the discrepancies $\mathfrak{u}-u$ and  $\mathfrak{p}-p$]\label{proposition:5.1}
Assume that the forcing terms $f$ and $g$ are sufficiently smooth  and the history function $\varPhi\in C^{\infty}([-q\tau,0];\mathcal{H_W})$ satisfies \eqref{adeq2}. 
Then, there exists a solution $(\mathfrak{u},\mathfrak{p})$ of the delay system \eqref{adeq1} which satisfies $\mathfrak{p}^{(q+1)}\in L^{\infty}((0,T); \mathcal{H_W})$. 
Moreover, the solutions of \eqref{weak} and \eqref{adeq1} only differ by a term of order $q$, i.e., for almost every $t\in[0,T]$
there holds 
\begin{equation*}
\|\mathfrak{u}(t)-u(t)\|_{\mathcal{V}}
+\|\mathfrak{p}(t)-p(t)\|_\mathcal{H_W}
+\|\mathfrak{p}(t)-p(t)\|_\mathcal{W}
\leqslant C\tau^q.
\end{equation*}
\end{proposition}
\begin{proof}
%Proceeding along the lines of the proof of Proposition A.4 in \cite{AMU:21},
%we see that $\mathfrak{p}^{(q+1)}\in L^{\infty}((0,T); \mathcal{H_W})$.
Let $e_p:=\mathfrak{p}-p$ and $e_u:=\mathfrak{u}-u$. 
Note that we have $e_u(0)=0$ and, due to the particular choice of the history function $\varPhi$ in \eqref{adeq2},
also $e_p(0)=0$.
Subtracting  \eqref{adeq1} from \eqref{weak}, we get
\begin{equation}\label{adeq4}
\begin{split}
a(e_u,v)-d(v,e_p)&=d(v,\hat{\mathfrak{p}}-\mathfrak{p}),\\
d(e_{u,t},w)+c(e_{p,t},w)+b(e_p,w)&=0,
\end{split}
\end{equation}
for all test functions $v\in\mathcal{V}$ and $w\in\mathcal{W}$.
Taking $v=e_{u,t}$ and $w=e_p$ in \eqref{adeq4}, we have
\begin{equation}\label{adeq5}
\begin{split}
\frac{1}{2}\frac{\dd}{\dd t}\|e_u\|_a^2
+\frac{1}{2}\frac{\dd}{\dd t}\|e_p\|_c^2+\|e_p\|_b^2
&\leqslant C_d\|e_{u,t}\|_\mathcal{V}\|\hat{\mathfrak{p}}-\mathfrak{p}\|_\mathcal{H_W}\\
&\leqslant \frac{1}{2}\|e_{u,t}\|^2_\mathcal{V}
+\frac{C_d^2}{2}\|\hat{\mathfrak{p}}-\mathfrak{p}\|^2_\mathcal{H_W}.
\end{split}
\end{equation}

The order of the $q$-step method $(\alpha,\gamma)$ is $q,$ i.e.,
\[\sum_{i=0}^{q-1}i^{\ell-1} \gamma_i= q^{\ell-1},\quad\ell=1,\dotsc,q.\]  
Therefore, in view of \eqref{extrapolation-delay}, by Taylor expanding about $t-q\tau$, we see that
%due to the order condition \eqref{order-new1} of the explicit scheme $(\alpha,\gamma)$,  
leading  terms of order up to $q-1$ cancel, and we obtain 
\begin{equation}\label{adeq6}
\begin{split}
\hat{\mathfrak{p}}(t)-\mathfrak{p}(t)
%=\sum\limits^{q-1}_{i=0}\gamma_i\mathfrak{p}(t-q\tau+i\tau)-\mathfrak{p}(t)
=\frac 1{(q-1)!}\Bigg [ \sum\limits^{q-1}_{i=0} &\gamma_i
\int_{t-q\tau}^{t-q\tau+i\tau}(t-q\tau+i\tau-s)^{q-1}\mathfrak{p}^{(q)}(s)\, \dd s\\
&-\int_{t-q\tau}^{t}(t-s)^{q-1}\mathfrak{p}^{(q)}(s)\, \dd s\Bigg ].
\end{split}
\end{equation}
Hence, we easily infer that
\begin{equation}\label{adeq7}
\|\hat{\mathfrak{p}}(t)-\mathfrak{p}(t)\|_\mathcal{H_W}
\leqslant C\tau^q\|\mathfrak{p}^{(q)}\|_{L^{\infty}((-q\tau,T);\mathcal{H_W})}.
\end{equation}
Integrating over $[0,t]$ in \eqref{adeq5} and  using the equivalence of norms,
% } ({\color{red}where? in the derivation of \eqref{adeq5}?}), 
%
\begin{equation}\label{norm-equiv}
\|\cdot\|_a^2\geqslant{c_a}\|\cdot\|_\mathcal{V}^2,\quad
\|\cdot\|_b^2\geqslant{c_b}\|\cdot\|_\mathcal{W}^2, \quad
\|\cdot\|_c^2\geqslant{c_c}\|\cdot\|_\mathcal{H_W}^2,
\end{equation}
we see that 
\begin{equation}\label{adeq8}
\begin{split}
&\|e_u(t)\|_\mathcal{V}^2
+\|e_p(t)\|_\mathcal{H_W}^2+\int_0^t\|e_p\|_\mathcal{W}^2\,\dd s\\
&\leqslant C\int_0^t\|e_{u,t}\|^2_\mathcal{V}\,\dd s
+C\tau^{2q}\|\mathfrak{p}^{(q)}\|^2_{L^{\infty}((-q\tau,T);\mathcal{H_W})}.
\end{split}
\end{equation}

Differentiation of the first equation in \eqref{adeq4} with respect to time yields
\begin{equation}\label{adeq9}
\begin{split}
a(e_{u,t},v)-d(v,e_{p,t})&=d(v,\hat{\mathfrak{p}}_t-\mathfrak{p}_t),\\
d(e_{u,t},w)+c(e_{p,t},w)+b(e_p,w)&=0,
\end{split}
\end{equation}
for all test functions $v\in\mathcal{V}$, $w\in\mathcal{W}$.
Taking $v=e_{u,t}$ and $w=e_{p,t}$ in \eqref{adeq9}, we have
\begin{equation}\label{adeq10}
\|e_{u,t}\|_a^2
+\|e_{p,t}\|_c^2+\frac{1}{2}\frac{\dd}{\dd t}\|e_p\|_b^2
\leqslant \frac{c_a}{2}\|e_{u,t}\|^2_\mathcal{V}
+\frac{C_d^2}{2c_a}\|\hat{\mathfrak{p}}_t-\mathfrak{p}_t\|^2_\mathcal{H_W}.
\end{equation}
In analogy to  \eqref{adeq7}, by replacing  $\mathfrak{p}$ by  $\mathfrak{p}_t$ in \eqref{adeq6}, 
we also have
\begin{equation*}
\|\hat{\mathfrak{p}}_t-\mathfrak{p}_t\|_\mathcal{H_W}
\leqslant C\tau^q\|\mathfrak{p}^{(q+1)}\|_{L^{\infty}((-q\tau,T);\mathcal{H_W})}.
\end{equation*}
Integrating over $[0,t]$ in \eqref{adeq10} and  using the equivalence of norms \eqref{norm-equiv},
%%
%\begin{equation*}
%\|\cdot\|_a^2\geqslant{c_a}\|\cdot\|_\mathcal{V}^2,\quad
%\|\cdot\|_b^2\geqslant{c_b}\|\cdot\|_\mathcal{W}^2, \quad
%\|\cdot\|_c^2\geqslant{c_c}\|\cdot\|_\mathcal{H_W}^2,
%\end{equation*}
%%
we get
\begin{equation}\label{adeq12}
\int_0^t\|e_{u,t}\|_\mathcal{V}^2\, \dd s
+\int_0^t\|e_{p,t}\|_\mathcal{H_W}^2\, \dd s+\|e_p\|_\mathcal{W}^2
\leqslant C\tau^{2q}\|\mathfrak{p}^{(q+1)}\|^2_{L^{\infty}((-q\tau,T);\mathcal{H_W})}.
\end{equation}
Combining \eqref{adeq8} and \eqref{adeq12}, and recalling that $\mathfrak{p}|_{[-q\tau,0]}=\varPhi$, we obtain
\begin{equation}\label{adeq13}
\begin{split}
&\|e_u\|^2_\mathcal{V}+\|e_p\|_\mathcal{H_W}^2+\|e_p\|_\mathcal{W}^2\\
&\leqslant C\tau^{2q}\Big(\|\mathfrak{p}^{(q)}\|^2_{L^{\infty}((-q\tau,T);\mathcal{H_W})}
+\|\mathfrak{p}^{(q+1)}\|^2_{L^{\infty}((-q\tau,T);\mathcal{H_W})}\Big)\\
&\leqslant C\tau^{2q}\Big(\|\mathfrak{p}^{(q)}\|^2_{L^{\infty}((0,T);\mathcal{H_W})}
+\|\mathfrak{p}^{(q+1)}\|^2_{L^{\infty}((0,T);\mathcal{H_W})}\\
&\quad\qquad+\|\varPhi^{(q)}\|^2_{L^{\infty}((-q\tau,0);\mathcal{H_W})}
+\|\varPhi^{(q+1)}\|^2_{L^{\infty}((-q\tau,0);\mathcal{H_W})}\Big).
\end{split}
\end{equation}
Proceeding along the lines of the proof of \cite[Proposition A.4]{AMU:21},
we see that $\mathfrak{p}$ and $\varPhi$ are smooth and 
their derivatives on the right-hand side of \eqref{adeq13} are bounded.
\end{proof}

\section{Consistency}\label{Se:Conc}
With $p_\star^\ell:=p(t_\ell)$ and $\mathfrak{p}_\star^\ell:=\mathfrak{p}(t_\ell)$ the nodal values of the solutions $p$ and $\mathfrak{p}$, 
 the consistency error $d^n$ of the fully implicit scheme  \eqref{full-a} for the solution $p$ of the original system \eqref{parabolic}, 
 and the consistency error $\tilde d^n$  of the implicit--explicit scheme  \eqref{abg3-a} for the solution $\mathfrak{p}$ 
 of the delay equation \eqref{adeq144},
  i.e., the amounts 
by which the exact solutions miss satisfying \eqref{full-a} and \eqref{abg3-a}, respectively, are given by 
\begin{equation}
\label{full-a-cons}
d^n:=(\mathcal{M}+\mathcal{C}){\dot p}_\star^n+ \B p_\star^n- g^n+\D\A^{-1}{\dot f}^n,\quad  n=q,\dotsc,N,
\end{equation}
and
\begin{equation}
\label{abg3-a-cons}
{\tilde d}^n:=\mathcal{C} \dot{\mathfrak{p}}_\star^{n}+\M\sum\limits^{q-1}_{j=0}\gamma_j
\dot{\mathfrak{p}}_\star^{n-q+j}
+ \B \mathfrak{p}_\star^n- g^n+\D\A^{-1}{\dot f}^{n},\quad  n=q,\dotsc,N.
\end{equation}

For convenience, we  introduce the quantities
\begin{equation}\label{defects}
\begin{alignedat}{2}
\eta^\ell &:={\dot p}_\star^\ell-p_t(t_\ell),\quad &&\ell=q,\dotsc,N,
\quad \xi^\ell :={\dot f}^\ell-f_t(t_\ell),\quad \ell=q,\dotsc,N,\\
\vartheta^\ell &:=\dot{\mathfrak{p}}_\star^\ell-\mathfrak{p}_t(t_\ell),
\quad &&\ell=0,\dotsc,N.
\end{alignedat}
\end{equation}
Notice that $\eta^\ell, \vartheta^\ell$, and $\xi^\ell$ are the defects of the discrete time derivatives,
%difference quotient of the $q$-step BDF scheme, 
for $p, \mathfrak{p}$, and $f,$ respectively. The  quantity $\vartheta^\ell$ will enter only 
in the consistency error of the implicit--explicit method.

With this notation, using the differential equations \eqref{parabolic} and \eqref{adeq144}, respectively, %and recalling the notation \eqref{defects}, 
we rewrite  the consistency errors $d^n$ and $\tilde d^n$  in the form
\begin{equation}
\label{full-a-cons1}
d^n=(\mathcal{M}+\mathcal{C})\eta^n+\D\A^{-1}\xi^n,\quad  n=q,\dotsc,N,
\end{equation}
and
\begin{equation}
\label{abg3-a-cons1}
{\tilde d}^n=\mathcal{C}\vartheta^n +\M\sum\limits^{q-1}_{j=0}\gamma_j\vartheta^{n-q+j}
+\D\A^{-1}\xi^n,\quad  n=q,\dotsc,N.
\end{equation}

\begin{lemma}[Consistency estimates]\label{Le:cons}
The consistency errors \eqref{full-a-cons1} and \eqref{abg3-a-cons1} are bounded by
\begin{equation}\label{cons-err-est}
\max_{q\leqslant n\leqslant N}\|d^n\|_{\mathcal{W}'} \leqslant C\tau^q, \quad 
\max_{q\leqslant n\leqslant N}\|\tilde d^n\|_{\mathcal{W}'} \leqslant \tilde C\tau^q,
\end{equation}
provided that the solutions $p$, $\mathfrak{p},$ and the forcing term $f$ are sufficiently regular.
\end{lemma}

\begin{proof}
The order of the $q$-step implicit method $(\alpha,\beta)$ is $q,$ i.e.,
\begin{equation}
\label{order-new}
\sum_{i=0}^qi^\ell \alpha_i=\ell q^{\ell-1},\quad \ell=0,1,\dotsc,q.  
\end{equation}
% 
%The order of the six-step BDF method is $6,$ i.e.,
%%
%\begin{equation}
%\label{order}
%\sum_{i=0}^qi^\ell \alpha_{i}=\ell q^{\ell-1},\quad \ell=0,1,\dotsc,q. 
%\end{equation}
%%
Therefore, by Taylor expanding about $t_{n-q}\geqslant0$, we see that
%due to the order conditions \eqref{order-new}, 
% of the implicit scheme $(\alpha,\beta)$, 
leading  terms of order up to $q-1$ cancel, and we obtain 
\begin{equation}\label{cons-eta}
\eta^{n}=\frac 1{q!}\Bigg [ \frac 1\tau \sum\limits^q_{i=0} \alpha_i
\int_{t_{n-q}}^{t_{n-q+i}}(t_{n-q+i}-s)^qp^{(q+1)}(s)\, \dd s-q\int_{t_{n-q}}^{t_{n}}(t_{n}-s)^{q-1}p^{(q+1)}(s)\, \dd s\Bigg ]
\end{equation}
and the corresponding relations for $\vartheta^n$ and $\xi^n$ with $p$ replaced by $\mathfrak{p}$ and $f,$ respectively.
We infer that the first estimate in \eqref{cons-err-est} for all $n$ as well as the  second estimate in \eqref{cons-err-est} 
for $n=2q,\dotsc,N,$ are valid.

For ${\tilde d}^q,\dotsc, {\tilde d}^{2q-1}$ in \eqref{abg3-a-cons1}, we need to estimate $\vartheta^0,\dotsc,\vartheta^{q-1}$.
Since $\mathfrak{p}$ is well defined in the interval $[-q\tau,T]$, Taylor expansion yields 
\begin{equation}\label{cons-err-est-neg}
\max_{q\leqslant n\leqslant 2q-1}\|\tilde d^n\|_{\mathcal{W}'} \leqslant \tilde C\tau^q.
\end{equation}
%
%{\color{blue}
In fact, \eqref{cons-err-est-neg} can be seen as follows. From \eqref{defects} and \eqref{discrete-deriv}, we obtain
\begin{equation*}
\vartheta^0=\dot{\mathfrak{p}}_\star^0-\mathfrak{p}_t(t_0)
=\dot{\mathfrak{p}}(t_0)-\mathfrak{p}_t(t_0)
=\frac1\tau\sum\limits^q_{i=0}\alpha_i\mathfrak{p}(t_{-q+i})-\mathfrak{p}_t(t_0).
\end{equation*}
By Taylor expanding about $t_{-q}$, we see that, in analogy to \eqref{cons-eta},
%due to the order conditions \eqref{order-new}, leading  terms of order up to $q-1$ cancel, and we obtain 
%
\begin{equation*}
\vartheta^{0}=\frac 1{q!}\Bigg [ \frac 1\tau \sum\limits^q_{i=0} \alpha_i
\int_{t_{-q}}^{t_{-q+i}}(t_{-q+i}-s)^q\mathfrak{p}^{(q+1)} (s)\, \dd s
-q\int_{t_{-q}}^{t_{0}}(t_{0}-s)^{q-1}\mathfrak{p}^{(q+1)}(s)\, \dd s\Bigg ].
\end{equation*}
Hence, we have
$\|\vartheta^{0}\|_{\mathcal{W}'}
\leqslant C\tau^q\|\mathfrak{p}^{(q+1)}\|_{L^{\infty}((-q\tau,0);\mathcal{H_W})}.$
Similarly,  $\vartheta^1,\dotsc,\vartheta^{q-1}$ can be estimated by Taylor expanding about $t_{-q+1},\dotsc,t_{-1}$, respectively.
Combining the estimates for $\xi^q,$ $\vartheta^0,\dotsc,\vartheta^{q}$, we obtain the estimate for ${\tilde d}^q$.
Also, ${\tilde d}^{q+1},\dotsc, {\tilde d}^{2q-1}$ can be estimated similarly as ${\tilde d}^q$,
and \eqref{cons-err-est-neg} is proved.
%}
%
%and the proof is complete.
%In view of \eqref{cons-eta} and the corresponding relations for  $\vartheta^n$ and $\xi^n$, the asserted consistency estimates \eqref{cons-err-est} 
%are easy consequences of the representations \eqref{full-a-cons1} and \eqref{abg3-a-cons1}.
\end{proof}

%Let us mention that only the order conditions of the implicit scheme  $(\alpha,\beta)$ are used in Lemma \ref{Le:cons};
%we used the order conditions of the implicit--explicit scheme  $(\alpha,\gamma)$  in Proposition \ref{proposition:5.1}.

%\newpage
\section{Convergence analysis}\label{Se:Conv}
Before we proceed, we recall the notion of the generating function of
an $n\times n$ Toeplitz  matrix $T_n$ as well as an auxiliary result, the Grenander--Szeg\H{o} theorem,
which plays a key role in our analysis.

\begin{definition}[{\cite[p.\ 13]{Chan:07}}; the generating function of a Toeplitz matrix]\label{De:gen-funct}
{\upshape Consider the $n \times n$ Toeplitz  matrix  %$T_n=(t_{ij})\in \C^{n,n}$
\[T_n=(t_{ij})_{i,j=1,\dotsc,n}\in \Co^{n,n}\]
with diagonal entries $t_0,$ subdiagonal entries
$t_1,$ superdiagonal entries $t_{-1},$ and so on, and $(n,1)$ and $(1,n)$ entries
$t_{n-1}$ and   $t_{1-n}$, respectively, i.e., the entries $t_{ij}=t_{i-j}, i,j=1,\dotsc,n,$ are constant along the diagonals of $T_n.$
 Let   $t_{-n+1},\dotsc, t_{n-1}$ be the Fourier coefficients of the trigonometric polynomial $g$
 of degree up to $n-1$, i.e.,
\begin{equation*}
  t_k=\frac{1}{2\pi}\int_{-\pi}^{\pi}g(x)\e^{-\ii kx}\, \mathrm{d} x,\quad k=1-n,\dotsc,n-1.
\end{equation*}
Then, $g(x)=\sum_{k=1-n}^{n-1} t_k\e^{\ii kx}$ is called  \emph{generating function} of $T_n$.}
\end{definition}

If the generating function $g$ is real-valued, then the matrix $T_n$ is Hermitian; if $g$ is real-valued and even,
then  $T_n$ is symmetric.

%Notice, in particular, that the generating function of a \emph{symmetric} band Toeplitz matrix of bandwidth $2m+1,$
%i.e., with $t_{m+1}=\dotsb=t_{n-1}=0,$ is a real-valued, even trigonometric polynomial, $f(x)= t_0+2t_1 \cos x+\dotsb+2t_m \cos (mx),$
%for all $n\geqslant m+1.$

\begin{lemma}[{\cite[pp.\ 13--14]{Chan:07}};  Grenander--Szeg\H{o} theorem]\label{Le:GS}
Let $T_n$ be a symmetric Toeplitz  matrix %as in Definition \ref{De:gen-funct} 
with generating function $g$.
%where $f$ is a $2\pi$-periodic continuous real-valued functions defined on $[-\pi,\pi]$.
Then, the smallest and largest eigenvalues  $\lambda_{\min}(T_n)$ and $\lambda_{\max}(T_n)$, respectively, of $T_n$
%Let $\lambda_{\min}(T_n)$ and $\lambda_{\max}(T_n)$ denote the smallest and largest eigenvalues of $T_n$, respectively.
%Then, they
are bounded as follows
\begin{equation*}
  g_{\min} \leqslant \lambda_{\min}(T_n) \leqslant \lambda_{\max}(T_n) \leqslant g_{\max},
\end{equation*}
with $g_{\min}$ and  $g_{\max}$  the minimum and maximum of $g$, respectively.
In particular, if  $g_{\min}$ is positive, then the symmetric matrix $T_n$ is positive
definite.
\end{lemma}

\subsection{Convergence of the fully implicit schemes}
Here, we derive error estimates for the fully implicit schemes by the energy technique.
We shall use the multipliers \eqref{mu},  \eqref{mu5}, and  \eqref{mu4} for the
six-, five-, and four-step schemes, respectively.  In contrast to the case of 
 the implicit--explicit schemes, for the fully implicit schemes we could have alternatively
 used the multipliers of \cite{NO} or \cite{AK:16} for the four- and five-step methods,
 and the multiplier of \cite{ACYZ:21} for the six-step method.

\begin{theorem}[Error estimates for fully implicit schemes]\label{theorem:fully}
Let $p(t_n)$ be the nodal values of the solution $p$ of the original system \eqref{parabolic} 
and $p^n$ satisfy the fully implicit $q$-step BDF scheme \eqref{full}, $q=4,5,6$. 
For sufficiently accurate  starting approximations $p^0,\dotsc,p^{q-1},$  such that
%%
%\begin{equation}\label{start-acc}
%\|p(t_j)-p^j\|_{\mathcal{H_Q}}+\tau^{1/2}\|p(t_j)-p^j\|_{\mathcal{Q}}\leqslant C\tau^q, \quad j=0,\dotsc,q-1,
%\end{equation}
%
\begin{equation}\label{start-acc}
\|p(t_j)-p^j\|_{\mathcal{H_W}}\leqslant C\tau^q,\quad j=0,\dotsc,q-1,
\end{equation}
and
\begin{equation}\label{start-acc6}
\tau^{1/2}\|p(t_j)-p^j\|_{\mathcal{W}}\leqslant C\tau^6, \quad j=3,4,5,
\end{equation}
for the six-step method,
\begin{equation}\label{start-acc5}
\tau^{1/2}\|p(t_j)-p^j\|_{\mathcal{W}}\leqslant C\tau^5, \quad j=3,4,
\end{equation}
for the five-step method, and 
\begin{equation}\label{start-acc4}
\tau^{1/2}\|p(t_3)-p^3\|_{\mathcal{W}}\leqslant C\tau^4
\end{equation}
for the four-step method, we have the optimal order error estimate
\begin{equation}\label{fi-acc}
\|p(t_n)-p^n\|_{\mathcal{H_W}}
\leqslant C\tau^q,\quad n=q,\dotsc,N.
\end{equation}
\end{theorem}

%\vspace*{-0.3cm}
\begin{proof}
For concreteness, we shall present the proof for the six-step method. For  the four- and five-step methods,
the proof proceeds along the same lines; the multipliers \eqref{mu4} and  \eqref{mu5} are used for these schemes.

Let $q=6$ and $e^n:=p(t_n)-p^n$ denote the error. Multiplying the consistency relation \eqref{full-a-cons} by $\tau$
and subtracting the scheme \eqref{full}, and then testing by $e^n-\mu_1e^{n-1}-\mu_2e^{n-2}-\mu_3e^{n-3}$,
with the multiplier in \eqref{mu},  we obtain
\begin{equation}\label{error1}
\Big ( (\mathcal{M}+\mathcal{C})\sum\limits^q_{i=0}\alpha_ie^{n-q+i},e^{n}-\sum_{j=1}^3\mu_j e^{n-j}\Big )
+ \tau B_{n}=\tau D_{n}
\end{equation}
with
\begin{equation*}
B_{n}:=\Big (\mathcal{B}e^n,e^{n}-\sum_{j=1}^3\mu_j e^{n-j}\Big )\quad\text{and}\quad
D_{n}:=\Big ( d^n,e^{n}-\sum_{j=1}^3\mu_j e^{n-j}\Big ).
\end{equation*}

We now consider the spectral decomposition of the operator $\mathcal{M}$ with respect to the inner product defined by
the bilinear form $c$. This means that we consider an orthonormal basis ${\tilde v}_i=\mathcal{C}^{1/2}v_i\in\mathcal{H_W}$ 
of eigenfunctions of the compact self-adjoint operator $\mathcal{C}^{-1/2}\mathcal{M}\mathcal{C}^{-1/2}$
corresponding to eigenvalues $0\leqslant \lambda_i\leqslant \omega;$ see \cite[Thm.\ 6.11]{Brezis:11}.
Then, obviously,
\begin{equation*}
\mathcal{M}v_i=\lambda_i\mathcal{C}v_i\quad\text{and}\quad
 (\mathcal{C}v_i,v_j)=\delta_{ij}.
\end{equation*}
With  the coefficients $\varepsilon_i^n:= (\mathcal{C}e^n,v_i)\in\Re,$ we have the orthonormal expansion  
$e^n=\sum\limits_{i=1}^\infty\varepsilon_i^n v_i$ and the Parseval relation
\begin{equation}\label{Parseval}
\|e^n\|_c^2=\sum_{i=1}\limits^\infty(\varepsilon_i^n)^2.
\end{equation}

Let ${\widetilde G}=(\tilde g_{jk})\in\Re^{q,q}$ be the positive definite symmetric matrix for the (implicit, with $m=0$) six-step BDF method
associated to the multiplier \eqref{mu} and denote by $\|\cdot\|_{\widetilde G}=({\widetilde G}\cdot,\cdot)^{1/2}$
the induced norm on $\Re^{q}$. Then, with $E_i^n:=\big (\varepsilon_i^{n-q+1},\dotsc,\varepsilon_i^{n}\big)^{\top}\in\Re^{q}$,
we have
\begin{equation*}
\left\|E_i^n\right\|_{\widetilde G}^2
=\sum_{j,k=1}^{q}{\tilde g}_{jk} \varepsilon_i^{n-q+j}\varepsilon_i^{n-q+k},
\end{equation*} 
%
%(G.A.: {\color{red}What is ${\widetilde G}$ here? Does it depend on a parameter? If yes, what is the range of
%the parameter? What is $\varepsilon^{n-q+i}$?})
%
%{\color{blue}
%The matrix ${\widetilde G}$ is corresponding Case I ($m=0$), which is independent on a parameter. $\varepsilon^{n-q+i}$ 
%is a typo mistake, we have changed into $\varepsilon_i^{n-q+j}$.
%}
and the first term on the left-hand side of \eqref{error1} can be estimated from below in the form
\begin{equation*}
\begin{split}
&\Big ( (\mathcal{M}+\mathcal{C})\sum\limits^q_{\ell=0}\alpha_\ell e^{n-q+\ell},e^{n}-\sum_{j=1}^3\mu_j e^{n-j}\Big )\\
&=\sum_{i=1}^\infty(\lambda_i+1)\Big (\sum\limits^q_{\ell=0}\alpha_\ell \varepsilon_i^{n-q+\ell}\Big ) \Big (\varepsilon_i^{n}-\sum_{j=1}^3\mu_j \varepsilon_i^{n-j}\Big )\\
&\geqslant\sum_{i=1}^\infty(\lambda_i+1)\big (\|E_i^n\|^2_{{\widetilde G}}-\|E_i^{n-1}\|^2_{{\widetilde G}}\big ).
\end{split}
\end{equation*}

With the norm $\|\cdot\|_{\lambda},$
\begin{equation}\label{norm-lambda}
\|E^m\|^2_{\lambda}=\sum_{i=1}^\infty(\lambda_i+1)\|E_i^m\|^2_{{\widetilde G}},
\end{equation}
the previous estimate takes the form
\begin{equation}\label{new-estimate}
\Big ( (\mathcal{M}+\mathcal{C})\sum\limits^q_{\ell=0}\alpha_\ell e^{n-q+\ell},e^{n}-\sum_{j=1}^3\mu_j e^{n-j}\Big )
\geqslant \|E^n\|^2_{\lambda}-\|E^{n-1}\|^2_{\lambda}.
\end{equation}
%
%Notice also that, due to the orthonormality of the  functions $\varphi_i$, we have $\left\|e^n\right\|_c^2=\sum_{i=1}\limits^\infty\left(\varepsilon_i^n\right)^2$.
In view of \eqref{new-estimate}, \eqref{error1} yields
$\|E^n\|^2_{\lambda}-\|E^{n-1}\|^2_{\lambda}+\tau B_{n}\leqslant \tau D_{n}.$
Summing here over $n$ from $n=q$ to $n=m$,  we obtain
\begin{equation}\label{2.888}
\|E^m\|^2_{\lambda}-\|E^{q-1}\|^2_{\lambda}+\tau \sum_{n=q}^mB_n\leqslant \tau \sum_{n=q}^mD_n.
\end{equation}
The sum on the right-hand side can be easily estimated by the generalized Cauchy--Schwarz inequality
and the arithmetic--geometric mean inequality with a suitable weight. 
Following the approach in \cite{ACYZ:21}, we next focus on
% It remains to estimate the sum $\sum_{n=6}^mA_n$ from below; we have
the estimation of the sum $B_q+\dotsb+B_m$ from below; we have
 \begin{equation}
\label{abg36}
\sum_{n=q}^mB_n=\sum_{n=q}^m\Big (\mathcal{B}e^n,e^{n}-\sum_{j=1}^3\mu_j e^{n-j}\Big ).
\end{equation}
First, motivated by the positivity of the function $g$ of \eqref{f}, to take advantage of the
positivity property  \eqref{pos-prop}, we introduce $\mu_0:=-91/100,$
and rewrite \eqref{abg36} as
 \begin{equation}
\label{abg37}
\sum_{n=q}^mB_n=\frac 9{100}\sum_{n=q}^m\|e^n\|_b^2+J_m\ \text{ with }\
J_m:=-\sum_{j=0}^3\mu_j \sum_{i=1}^{m-5}\left(\mathcal{B} e^{5+i},e^{5+i-j} \right).
\end{equation}

Our next task is to rewrite $J_m$ in a form that will enable us to estimate
it from below in a desired way.
To this end, we introduce the lower triangular Toeplitz matrix $L=(\ell_{ij})\in \R^{m-5,m-5}$ with entries
$\ell_{i,i-j}=-\mu_j, ~ j=0,1,2,3, ~ i=j+1,\dotsc,m-5,$
and all other entries equal zero.
With this notation, we have
\begin{equation}\label{abg38}
\begin{split}
\sum_{i,j=1}^{m-5}&\ell_{ij} (\mathcal{B} e^{5+i},e^{5+j} )=-\sum_{j=0}^3\mu_j \sum_{i=j+1}^{m-5}(\mathcal{B} e^{5+i},e^{5+i-j})\\
&=J_m+(\mathcal{B} e^6,\mu_1e^5+\mu_2e^4+\mu_3e^3)+(\mathcal{B} e^7,\mu_2e^5+\mu_3e^4)+\mu_3(\mathcal{B} e^8,e^5).\\
%&+(\mathcal{B} e^7,\mu_2e^5+\mu_3e^4)+\mu_3(\mathcal{B} e^8,e^5).
\end{split}
\end{equation}
At this point we shall use the positivity property \eqref{pos-prop}  to show that the term on the left-hand side
of \eqref{abg38} is nonnegative and then obtain a suitable lower bound for $J_m.$ Indeed, the symmetric part
\[L_s:=(L+L^\top)/2\]
of the matrix $L$ is a symmetric seven-diagonal Toeplitz matrix and its
 generating function $g,$ see \eqref{f}, is positive.
Hence, according to the Grenander--Szeg\H{o} theorem, see Lemma \ref{Le:GS},
the Toeplitz matrix $L_s$ is positive definite. Consequently, since
\[(Lx,x)=(L_sx,x)\quad \forall x\in \R^{m-5},\]
the matrix $L$ is also positive definite.
%Now, in view of the positivity of the generating function $f,$ see \eqref{f}, of the symmetric part %$L_s:=(L+L^\top)/2$
%%
%\[L_s:=(L+L^\top)/2\]
%%
%of the matrix $L,$ the Grenander--Szeg\H{o} theorem, see Lemma \ref{Le:GS},
%ensures positive definiteness of $L_s,$ and thus also of $L$ itself, since %$(Lx,x)=(L_sx,x)$ for $x\in \Re^{m-5}.$
%%
%\[(Lx,x)=(L_sx,x)\quad \forall x\in \Re^{m-5}.\]
%%
Therefore, the expression on the left-hand side of \eqref{abg38} is nonnegative; thus,
\eqref{abg38} yields the desired estimate for $J_m$ from below,
\begin{equation}
\label{abg39}
J_m\geqslant-(\mathcal{B} e^6,\mu_1e^5+\mu_2e^4+\mu_3e^3)
-(\mathcal{B} e^7,\mu_2e^5+\mu_3e^4)-\mu_3(\mathcal{B} e^8,e^5 ).
\end{equation}

From \eqref{2.888}, \eqref{abg37}   and \eqref{abg39},  we have
\begin{equation}\label{est-essent}
\begin{split}
\|E^m\|^2_{\lambda}&+\frac 9{100}\tau \sum_{n=q}^m\|e^n\|_b^2
\leqslant
\|E^{q-1}\|^2_{\lambda}+\tau \sum_{n=q}^mD_n\\
&+\tau(\mathcal{B} e^6,\mu_1e^5+\mu_2e^4+\mu_3e^3)+\tau (\mathcal{B} e^7,\mu_2e^5+\mu_3e^4 )+\tau\mu_3(\mathcal{B} e^8,e^5).
\end{split}
\end{equation}
Now, from  \eqref{norm-lambda} and  \eqref{Parseval}, with $c_{\widetilde G}$ and $C_{\widetilde G}$ the smallest and largest eigenvalues 
of the matrix $\widetilde G,$ we obtain
\[\|E^m\|^2_{\lambda}\geqslant \sum_{i=1}^\infty\|E_i^m\|^2_{\widetilde G}\geqslant c_{\widetilde G}\sum_{i=1}^\infty\left(\varepsilon_i^m\right)^2,\]
whence,
\begin{equation}\label{new-1}
\|E^m\|^2_{\lambda}\geqslant c_{\widetilde G}\|e^m\|_c^2,
\end{equation}
and, analogously,
\begin{equation}\label{new-2}
\|E^{q-1}\|^2_{\lambda}\leqslant (1+\omega) C_{\widetilde G}\sum_{j=0}^{q-1}\|e^j\|_c^2.
\end{equation}
%
%%
%\begin{equation*}
%\begin{aligned}
%\sum_{i=1}^\infty\|E_i^m\|^2_{{\widetilde G}(\lambda_i)}
%&\geqslant c_{\widetilde G}\sum_{i=1}^\infty\left(\varepsilon_i^m\right)^2
%=c_{\widetilde G}\left\|e^m\right\|_c^2,\\
%\sum_{i=1}^\infty\|E_i^{q-1}\|^2_{{\widetilde G}(\lambda_i)}
%&\leqslant C_{\widetilde G}\sum_{i=1}^\infty\sum_{j=0}^{q-1}\left(\varepsilon_i^j\right)^2
%=C_{\widetilde G}\sum_{j=0}^{q-1}\|e^j\|_c^2.
%\end{aligned}
%\end{equation*}
%

Furthermore, the terms involving  the starting approximations can be estimated
by elementary inequalities in the form
\begin{equation}\label{B-estimate}
|(\mathcal{B} e^i,e^j)|\leqslant \delta_1C_b \|e^i\|_{\mathcal{W}}^2+\frac {C_b}{4\delta_1}\|e^j\|_{\mathcal{W}}^2,\quad i=6,7,8, \ \
j=3,4,5,
 \end{equation}
%\text{for}\quad i>j\]
%
with sufficiently small $\delta_1$. 
The term $D_n$ can be estimated by the Cauchy--Schwarz and arithmetic--geometric mean  inequalities; we obtain
\begin{equation}\label{D-estimate}
|D_n|\leqslant \frac 1{\delta_2}\|d^n\|^2_{\mathcal{W}'}+\delta_2\Big (\|e^n\|_{\mathcal{W}}^2+\sum_{j=1}^3|\mu_j|^2 \|e^{n-j}\|^2_{\mathcal{W}}\Big)
\end{equation}
%\text{for}\quad i>j\]
%
with sufficiently small $\delta_2$.

Utilizing \eqref{new-1}, \eqref{new-2}, \eqref{B-estimate}, and \eqref{D-estimate}, we infer from 
\eqref{est-essent} that
\begin{equation*}
\|e^m\|_{\mathcal{H_W}}^2+\tau \sum_{n=6}^m\|e^n\|_{\mathcal{W}}^2
\leqslant C\sum_{j=0}^5\left(\|e^j\|_{\mathcal{H_W}}^2
+\tau\|e^j\|_{\mathcal{W}}^2\right)
+C\tau\sum_{n=6}^m\|d^n\|^2_{\mathcal{W}'},
\ \ m=6,\dotsc,N.
\end{equation*}
In view of  \eqref{start-acc} and  \eqref{cons-err-est}, the expression on the right-hand side
is of order $O(\tau^{12})$ and the asserted estimate \eqref{fi-acc} for $q=6$ follows.
\end{proof}

\begin{theorem}[Error estimates for $u$]\label{theorem:fully-u}
Let $u(t_n)$ be the nodal values of the solution $u$ of the original system \eqref{weak}
and $u^n$ satisfy the fully implicit scheme  \eqref{implicit}. For sufficiently accurate starting approximations
$p^0,\dotsc,p^{q-1},$ satisfying \eqref{start-acc}--\eqref{start-acc4},
we have the error estimate
\begin{equation}\label{fi-acc-u}
\|u(t_n)-u^n\|_\mathcal{V}\leqslant C\tau^q,\quad n=q,\dotsc,N.
\end{equation}
\end{theorem}
\begin{proof}
Considering the difference between the first equations of \eqref{weak} and \eqref{implicit}, we obtain
\begin{equation*}
a(u(t_n)-u^n,v)-d(v,p(t_n)-p^n)=0 \quad \forall v\in\mathcal{V}.
\end{equation*}
For the test function $v=u(t_n)-u^n$, we get
\begin{equation*}
\|u(t_n)-u^n\|_\mathcal{V}^2
\leqslant \frac{C_d}{c_a}\|u(t_n)-u^n\|_\mathcal{V}\,
\|p(t_n)-p^n\|_{\mathcal{H_W}},
\end{equation*}
whence,  
$\|u(t_n)-u^n\|_\mathcal{V} \leqslant C\|p(t_n)-p^n\|_{\mathcal{H_W}}.$
Thus, the asserted estimate \eqref{fi-acc-u} follows from \eqref{fi-acc}.
\end{proof}

\begin{remark}[On the requirements on the starting approximations]\label{RE:start-approx}
{\upshape The accuracy requirements \eqref{start-acc6},  \eqref{start-acc5}, and  \eqref{start-acc4}
on the starting approximations are actually not necessary; they are technical assumptions 
due to the energy technique. For instance,  \eqref{start-acc6} is due to the nonvanishing
 components $\mu_1, \mu_2,$ and $\mu_3$ of the multiplier in \eqref{mu}.
%%
%\begin{equation}\label{start-acc-n}
%\|p(t_j)-p^j\|_{\mathcal{Q}}\leqslant C\tau^{q-1/2}, \quad j=3,4,5,
%\end{equation}
%%
%cf.\ \eqref{start-acc}, 
%on the starting approximations for the six-step
%BDF method is actually not necessary. It is a technical assumption
%due to the energy technique and the nonvanishing components
%$\mu_1, \mu_2,$ and $\mu_3$ of the multiplier in \eqref{mu}.

A combination of the Fourier and spectral stability techniques
allows us to establish the optimal order error estimate \eqref{fi-acc},
in a unified way for all BDF methods, under the milder accuracy assumption 
\eqref{start-acc}
%
%\begin{equation}\label{start-acc-milder}
%\|p(t_j)-p^j\|_{\mathcal{H_Q}}\leqslant C\tau^q, \quad j=0,\dotsc,q-1,
%\end{equation}
%
on the starting approximations;
see \cite[Remark 7.2]{AC} and  \cite[Remark 2.1, (2.25)]{AKK}. 
For more details on the  Fourier and spectral stability technique,
see, for instance, \cite{ACM2,A1,A2} and references therein.

We employed the energy technique here since it is applicable
also to the more interesting case of the implicit--explicit BDF methods.
In contrast, the Fourier and spectral stability technique does not seem
to be applicable in this case; this is due to the fact that the time 
derivative $p_t$ is not discretized in the same way in all terms;
the coefficients of the polynomial $\alpha\in \P_q$ are used
in the discretization of part of it, namely, in the term $\mathcal{C} p_t,$  
while the coefficients of the polynomial $\tilde \alpha\in \P_{2q-1}$ are used
in the discretization of the other part of it, namely, in the term $\M p_t$;  
see the first and second terms on the left-hand side of \eqref{abg3}.}
\end{remark}

\subsection{Convergence of the implicit--explicit schemes}
We next derive error estimates for the implicit--explicit schemes by the energy technique.  
\begin{theorem}[Error estimates for  implicit--explicit schemes]\label{theorem:ie-p}
Let $\mathfrak{p}(t_n)$ be the nodal values of the solution $\mathfrak{p}$ of the delay system \eqref{adeq144}
and $p^n$ satisfy the implicit--explicit $q$-step BDF scheme \eqref{numerical}, $q=4,5,6.$
Assume that $\omega\leqslant 1/(2^q-1)$.  Then,   we have 
\begin{equation*}
\|\mathfrak{p}(t_n)-p^n\|^2_{\mathcal{H_W}}
\leqslant C\tau^{2q}
+C\sum_{j=0}^{q-1}\Big(\|\mathfrak{p}(t_j)-p^j\|_{\mathcal{H_W}}^2
+\tau\|\mathfrak{p}(t_j)-p^j\|_{\mathcal{W}}^2\Big ),\  n=q,\dotsc,N.
\end{equation*}
%%
%it holds that
%\begin{equation*}
%\left\|p(t_n)-p^n\right\|_{\mathcal{H_Q}}^2
%\leqslant C\sum_{j=0}^{q-1}\left(\|p(t_j)-p^j\|_{\mathcal{H_Q}}^2
%+\tau\|p(t_j)-p^j\|_{\mathcal{Q}}^2\right) +C\tau^{2q}.
%\end{equation*}
\end{theorem}

\begin{proof}
Again, for concreteness, we shall present the proof for $q=6.$ %the six-step method.

Let $e^n:=\mathfrak{p}(t_n)-p^n$ denote the error. Multiplying the consistency relation \eqref{abg3-a-cons}
 by $\tau$ and subtracting the scheme \eqref{numerical}, and then testing by $e^n-\mu_1e^{n-1}-\mu_2e^{n-2}-\mu_3e^{n-3}$,
with the multiplier in \eqref{mu},  we obtain
%We define the error and the  consistency error by
%\begin{equation*}
%\begin{split}
%e^n:=&p(t_n)-p^n,\quad d^n:=\frac{1}{\tau}\sum_{i=0}^{q}\alpha_ip(t_{n-q+i})- p_t(t_n), \\ 
%\eta^n:=&\frac{1}{\tau}\sum\limits^{2q-1}_{i=0}\tilde \gamma_ip(t_{n-2q+i})- p_t(t_n),\quad\xi^n:=\frac{1}{\tau}\sum_{i=0}^{q}\alpha_if(t_{n-q+i})-f_t(t_n).
%\end{split}
%\end{equation*}
%Considering the difference between \eqref{numerical} and \eqref{parabolic} and testing with $e^n-\mu_1e^{n-1}-\mu_2e^{n-2}-\mu_3e^{n-3}$ in \eqref{mu}, we obtain
%
\begin{equation}\label{error}
\Big (\mathcal{C} \sum\limits^{2q}_{i=0}\tilde \alpha_ie^{n-2q+i}+\M\sum\limits^{2q-1}_{i=0}\hat \alpha_ie^{n-2q+i},e^{n}-\sum_{j=1}^3\mu_j e^{n-j}\Big )+ \tau B_{n}= \tau \widetilde{D}_{n}
\end{equation}
$n=q,\dotsc,N,$ with
\begin{equation*}
B_{n}:=\Big (\mathcal{B}e^n,e^{n}-\sum_{j=1}^3\mu_j e^{n-j}\Big )\quad\text{and}\quad
\widetilde{D}_{n}:=\Big ({\tilde d}^n,e^{n}-\sum_{j=1}^3\mu_j e^{n-j}\Big ).
\end{equation*}

Let $G(\lambda_i)=(g_{jk}(\lambda_i))\in\Re^{2q,2q}$ be the positive definite symmetric matrices for the 
(implicit--explicit, with $m\in[0,1/63]$) six-step BDF method, for $m=\lambda_i,$ associated to the multiplier 
\eqref{mu} and denote by $\|\cdot\|_{G(\lambda_i)}=({G(\lambda_i)}\cdot,\cdot)^{1/2}$ the induced norm on $\Re^{2q}$.
%The matrices ${G(\lambda_i)}$ depend on the parameter $\lambda_i\in[0,1/63]$.
Let $\widetilde{G}=(\tilde{g}_{jk}),\widehat{G}=(\hat{g}_{jk})\in\Re^{2q,2q}$ be the positive definite symmetric matrices 
for the extreme cases $m=0$ and $m=1/63$,
respectively, for the six-step BDF method, associated to the multiplier \eqref{mu}. 
Since $\tilde \alpha+\lambda_i \hat \alpha=(1-\frac {\lambda_i} {\omega^{\star}})\tilde \alpha +
\frac {\lambda_i} {\omega^{\star}} (\tilde \alpha+\omega^{\star} \hat \alpha),$
it is easily seen that
%n fact, we have
%
\begin{equation}\label{G-matrices}
g_{jk}(\lambda_i)=\Big (1-\frac {\lambda_i} {\omega^{\star}}\Big )\tilde{g}_{jk} +\frac {\lambda_i} {\omega^{\star}} \hat{g}_{jk},\quad \omega^{\star}=1/63.
\end{equation}
Then, with $\mathcal{E}_i^n:=\big (\varepsilon_i^{n-2q+1},\ldots,\varepsilon_i^{n}\big )^{\top}\in\Re^{2q}$, we have
\begin{equation*}
\left\|\mathcal{E}_i^n\right\|_{G(\lambda_i)}^2
=\sum_{j,k=1}^{2q}g_{jk}(\lambda_i)\varepsilon_i^{n-2q+j}\varepsilon_i^{n-2q+k},
\end{equation*} 
and the first term on the left-hand side of \eqref{error} can be estimated from below in the form
\begin{equation*}
\begin{split}
&\Big (\mathcal{C} \sum\limits^{2q}_{\ell=0}\tilde \alpha_\ell e^{n-2q+\ell}+\M\sum\limits^{2q-1}_{\ell=0}\hat \alpha_\ell e^{n-2q+\ell},e^{n}-\sum_{j=1}^3\mu_j e^{n-j}\Big )\\
&=\sum_{i=1}^\infty\Big (\sum\limits^{2q}_{\ell=0}\left(\tilde \alpha_\ell+\lambda_i\hat \alpha_\ell\right)\varepsilon_i^{n-2q+\ell}\Big )\Big (\varepsilon_i^{n}
-\sum_{j=1}^3\mu_j \varepsilon_i^{n-j}\Big )\\
&\geqslant\sum_{i=1}^\infty\Big (\left\|\mathcal{E}_i^n\right\|^2_{G(\lambda_i)}-\|\mathcal{E}_i^{n-1}\|^2_{G(\lambda_i)}\Big ).
\end{split}
\end{equation*}
With the norm $\|\cdot\|_\lambda,$
\[\|\mathcal{E}^m\|_\lambda^2=\sum_{i=1}^\infty\left\|\mathcal{E}_i^m\right\|^2_{G(\lambda_i)},\]
the previous estimate and  \eqref{error} yield
\begin{equation*}\label{2.87}
\|\mathcal{E}^n\|_\lambda^2-\|\mathcal{E}^{n-1}\|_\lambda^2
+\tau B_{n}\leqslant \tau \widetilde{D}_{n}.
\end{equation*}
Summing here over $n,$ from $n=q$ to $n=m$,  we have
\begin{equation}\label{2.88}
\|\mathcal{E}^m\|_\lambda^2+\tau \sum_{n=q}^mB_n\leqslant \|\mathcal{E}^{q-1}\|_\lambda^2+\tau \sum_{n=q}^m\widetilde{D}_n.
\end{equation}

From \eqref{2.88}, \eqref{abg37},   and \eqref{abg39}, we obtain
\begin{equation*}
\begin{split}
\|\mathcal{E}^m\|_\lambda^2+\frac 9{100}\tau \sum_{n=q}^m\|e^n\|_b^2
&\leqslant \|\mathcal{E}^{q-1}\|_\lambda^2
+\tau \sum_{n=q}^m\widetilde{D}_n+\tau(\mathcal{B} e^6,\mu_1e^5+\mu_2e^4+\mu_3e^3 )\\
&+\tau (\mathcal{B} e^7,\mu_2e^5+\mu_3e^4 )+\tau\mu_3(\mathcal{B} e^8,e^5 ).
\end{split}
\end{equation*}
%
%Now, using \eqref{Parseval}, with $c_G$ and $C_G$ the smallest and largest eigenvalues of all matrices $G(\lambda_i)$
%({\color{blue}make this precise; refer to \eqref{G-matrices}}), 

%{\color{red}
Let ($c_{G(\lambda_i)}, C_{G(\lambda_i)}$), ($c_{\widetilde{G}}, C_{\widetilde{G}}$), and ($c_{\widehat{G}}, C_{\widehat{G}}$) 
denote the smallest and largest eigenvalues of all matrices $G(\lambda_i),\widetilde{G}$, and $\widehat{G}$, respectively.
According to Weyl's theorem, \cite[Theorem 1.4]{Chan:07}, and \eqref{G-matrices}, we have
\begin{equation*}
\begin{split}
c_{G(\lambda_i)}
\geqslant\Big (1-\frac {\lambda_i} {\omega^{\star}}\Big )c_{\widetilde{G}} +\frac {\lambda_i} {\omega^{\star}}c_{\widehat{G}}
=:c_{\lambda_i},\quad
C_{G(\lambda_i)}
\leqslant\Big (1-\frac {\lambda_i} {\omega^{\star}}\Big )C_{\widetilde{G}} +\frac {\lambda_i} {\omega^{\star}}C_{\widehat{G}}
=:C_{\lambda_i}.
\end{split}
\end{equation*}
%
%}
Now, with
$c_{\star}:= \min_{\lambda_i\in[0,\omega^{\star}]}c_{\lambda_i}$ and
$C_{\star}:= \max_{\lambda_i\in[0,\omega^{\star}]}C_{\lambda_i}$,
using \eqref{Parseval}, we obtain
\begin{equation*}
\begin{split}
\|\mathcal{E}^m\|_\lambda^2
&\geqslant\sum_{i=1}^\infty c_{G(\lambda_i)}
\left(\varepsilon_i^m\right)^2
\geqslant \min_{\lambda_i\in[0,\omega^{\star}]}c_{\lambda_i} \sum_{i=1}^\infty\left(\varepsilon_i^m\right)^2
=c_{\star}\left\|e^m\right\|_c^2,\\
 \|\mathcal{E}^{q-1}\|_\lambda^2
&\leqslant\sum_{i=1}^\infty C_{G(\lambda_i)}
\sum_{j={-q}}^{q-1}(\varepsilon_i^j)^2
\leqslant\max_{\lambda_i\in[0,\omega^{\star}]}C_{\lambda_i}
\sum_{i=1}^\infty\sum_{j={-q}}^{q-1}(\varepsilon_i^j)^2
= C_{\star}\sum_{j=0}^{q-1}\|e^j\|_c^2.
\end{split}
\end{equation*}
Notice that  $e^{-q},\dotsc,e^{-1}$ do not enter in the above estimate of $\mathcal{E}^{q-1}$ because they vanish;
see  \eqref{adeq2}.

%From \eqref{adeq2}, notice that  $e^{-q},\dotsc,e^{-1}$ do not enter in the above estimate of $\mathcal{E}_i^{q-1}$ because they vanish.
%

Furthermore, the terms involving the starting approximations and ${\widetilde D}_n$
can be estimated as in the  proof of Theorem \ref{theorem:fully}. 
The proof is complete.
\end{proof}

\begin{theorem}[Error estimates for $\mathfrak{u}$]\label{theorem:ie-u}
Let $\mathfrak{u}(t_n)$ be the nodal values of the solution $\mathfrak{u}$ of the delay system \eqref{adeq1}
and $u^n$ satisfy the implicit--explicit $q$-step BDF scheme \eqref{explicit}, $q=4,5,6$. Assume that $\omega\leqslant1/(2^q-1)$.
Then, we have the error estimate
\begin{equation*}
\|\mathfrak{u}(t_n)-u^n\|_\mathcal{V}^2
\leqslant C\tau^{2q}
+C\sum_{j=0}^{q-1}\Big(\|\mathfrak{p}(t_j)-p^j\|_{\mathcal{H_W}}^2
+\tau\|\mathfrak{p}(t_j)-p^j\|_{\mathcal{W}}^2\Big),\ \ n=q,\dotsc,N.
\end{equation*}
\end{theorem}

\begin{proof}
Subtracting the  first equation of \eqref{explicit} from the first equation
of \eqref{adeq1} yields
%Considering the difference between the first equations of \eqref{adeq1} and \eqref{explicit}, we obtain
%
\begin{equation*}
a(\mathfrak{u}(t_n)-u^n,v)-d(v,\hat{\mathfrak{p}}(t_n)-\hat{p}^n)=0 \quad \forall v\in\mathcal{V}.
\end{equation*}
For the test function $v=\mathfrak{u}(t_n)-u^n$, we get
\begin{equation*}
\|\mathfrak{u}(t_n)-u^n\|_\mathcal{V}^2
\leqslant \frac{C_d}{c_a}\|\mathfrak{u}(t_n)-u^n\|_\mathcal{V}
\|\hat{\mathfrak{p}}(t_n)-\hat{p}^n\|_{\mathcal{H_W}},
\end{equation*}
whence
\begin{equation*}
\|\mathfrak{u}(t_n)-u^n\|_\mathcal{V}
\leqslant C\|\hat{\mathfrak{p}}(t_n)-\hat{p}^n\|_{\mathcal{H_W}}
\leqslant C\sum\limits^{q-1}_{j=0}\gamma_j
\|\mathfrak{p}(t_{n-q+j})-p^{n-q+j}\|_{\mathcal{H_W}}.
\end{equation*}
Thus, the asserted estimate follows immediately from Theorem \ref{theorem:ie-p}.
\end{proof}

%
%Using the triangle inequality, we have
%%
%\begin{equation*}
%\begin{split}
%\|u_\star^n-u^n\|_\mathcal{V}
%&\leqslant C\|p_\star^n-{\hat p}_\star^n\|_{\mathcal{H_Q}}
%+C\|{\hat p}_\star^n-{\hat p}^n\|_{\mathcal{H_Q}}\\
%&\leqslant C\big\|p(t_n)-\sum\limits^{q-1}_{i=0} \gamma_ip(t_{n-q+i})\big\|_{\mathcal{H_Q}}
%+C\sum_{\ell=0}^{q-1}|\gamma_{\ell}|\,\|p(t_{n-q+\ell})-p^{n-q+\ell}\|_{\mathcal{H_Q}}.
%\end{split}
%\end{equation*}
%%
%Combining a similar estimate of $\theta^n$ and
%Theorem \ref{theorem:semi}, the desired result is obtained.

%
%
Combining Proposition \ref{proposition:5.1} and Theorems \ref{theorem:ie-p} and \ref{theorem:ie-u}, we obtain 
error estimates for the implicit--explicit schemes.

\begin{theorem}[Error estimates for implicit--explicit schemes]\label{theorem:semi}
Let $p(t_n)$ be the nodal values of the solution $p$ of the original system \eqref{parabolic} 
and $p^n$ satisfy the implicit--explicit $q$-step BDF scheme \eqref{numerical}, $q=4,5,6$.
Assume that $\omega\leqslant1/(2^q-1)$.  Then, for sufficiently accurate starting approximations
$p^0,\dotsc,p^{q-1},$ 
\begin{equation*}
\|p(t_j)-p^j\|_{\mathcal{H_W}}+\tau^{1/2}\|p(t_j)-p^j\|_{\mathcal{W}}
\leqslant C\tau^q, \quad j=0,\dotsc,q-1,
\end{equation*}
we have the optimal order error estimate
\begin{equation*}
\|u(t_n)-u^n\|_\mathcal{V}+\|p(t_n)-p^n\|_{\mathcal{H_W}}
\leqslant C\tau^q,\quad n=q,\dotsc,N.
\end{equation*}
%%
%it holds that
%\begin{equation*}
%\left\|p(t_n)-p^n\right\|_{\mathcal{H_Q}}^2
%\leqslant C\sum_{j=0}^{q-1}\left(\|p(t_j)-p^j\|_{\mathcal{H_Q}}^2
%+\tau\|p(t_j)-p^j\|_{\mathcal{Q}}^2\right) +C\tau^{2q}.
%\end{equation*}
\end{theorem}

%\newpage
\section{Numerical results}\label{Se:numerics}
In this section, we present numerical examples to demonstrate the convergence result stated in Theorem \ref{theorem:semi} and highlight 
the necessity of imposing a weak coupling condition.
\subsection{Poroelastic example}
For the sake of brevity, we employ the implicit--explicit and fully implicit six-step BDF schemes to \eqref{system}
with  $\varOmega=(-1,1)^2, T=1$, and  homogeneous Dirichlet boundary conditions.

We numerically verified the  theoretical results including convergence orders. In  space, we discretized by the spectral collocation method  with
the Cheby\-shev--Gauss--Lobatto points.
%
%\[u_I^{n}(x,y)=\sum^{N_x}_{i=0}\sum^{N_y}_{j=0}u_{ij}^{n}\ell_i(x)\ell_j(y),
%\quad \ell_i(x)=\prod_{\substack{j=0\\ j\ne i}}^{N_x}\frac{x-x_j}{x_i-x_j},\]
%
%where  $u_{ij}^{n}:=u_I^{n}(x_i,y_j)$ at the mesh points  $(x_i,y_j)$.
%Here, $-1=x_0<x_1<\dotsb<x_{N_x}=1$ and $-1=y_0<y_1<\dotsb<y_{N_y}=1$
%are nodes of Lobatto quadrature rules.
In order to test the temporal error, we fix $N_x = N_y = 20$;
the spatial error is negligible  since the spectral collocation method converges exponentially;
see, e.g., \cite[Theorem 4.4, \textsection{4.5.2}]{STW:2011}.

\begin{example}\label{ex:6.1}
{\upshape
Here, the initial value and the forcing term were
chosen such that the exact solution of equation \eqref{system} is
\begin{equation*}
u(x,y,t)=(t^7+1)
\begin{pmatrix}
\left(\cos(\pi x)+1\right)\sin(\pi y)\\
\sin(\pi x)\left(\cos(\pi y)+1\right)
\end{pmatrix},
\quad 
p(x,y,t)=(t^7+1)\sin(\pi x)\sin(\pi y).
\end{equation*}
The poroelasticity parameters are chosen as $\eta=0.3,\mu=0.3,\lambda=0.3,M=0.1,\kappa=0.05$;
then, the coupling strength  $\omega=0.015<1/63$ satisfies the condition in  Theorem \ref{theorem:semi}.
%For this case, 
We present in Table \ref{table:1} the errors as well as the corresponding convergence orders (rates).
%

%\vspace*{-0.4cm}
\begin{table}[!ht]
\begin{center}
{\small
\caption{Errors and convergence orders with $\omega<1/63.$}
% \vspace{5pt}
\begin{tabular}{|c|c|c|c|c|}\cline{1-5} 
\multicolumn{5}{|c|}{Implicit--explicit six-step scheme} \\  
\hline 
$\tau{\vphantom{\sum^{\sum^\sum}}}$ & $\|u(T)-u^N\|_{\mathcal{V}}$   &   Rate  \ \  & $\|p(T)-p^N\|_{\mathcal{H_W}}$     & Rate    \\*[2pt]
\hhline{|=====|}
  1/50&     1.2218e-07  &         &4.9606e-08  &             \\
  1/100&    1.9896e-09  &5.9404   &7.4826e-10  &6.0508      \\
  1/150&    1.7700e-10  &5.9673   &6.4951e-11  &6.0279        \\
  1/200&    3.1689e-11  &5.9795   &1.1535e-11  &6.0076          \\ 
  \hline
\multicolumn{5}{|c|}{Fully implicit six-step scheme} \\ \hline
$\tau{\vphantom{\sum^{\sum^\sum}}}$ &        $\|u(T)-u^N\|_{\mathcal{V}}$    &   Rate  \ \  & $\|p(T)-p^N\|_{\mathcal{H_W}}$     & Rate   \\*[2pt]
\hhline{|=====|}
  1/50&     1.7802e-08  &         &4.6881e-08  &            \\
  1/100&    2.6747e-10  &6.0565   &7.0436e-10  &6.0566      \\
  1/150&    2.3181e-11  &6.0318   &6.1035e-11  &6.0322        \\
  1/200&    4.1175e-12  &6.0070   &1.0825e-11  &6.0120          \\ 
  \hline
    \end{tabular}\label{table:1}
    }
  \end{center}
\end{table}

Next, we choose the poroelasticity parameters $\eta=0.6,\mu=0.6,\lambda=0.6,M=0.1,\kappa=0.05$;
then,  the coupling strength is $\omega=0.03>1/63,$ whence the condition in Theorem \ref{theorem:semi} is violated.
We present  the errors as well as the corresponding convergence orders (rates)  in Table \ref{table:2}.

\begin{table}[!ht]
\begin{center}
{\small
\caption{Errors and convergence orders with $\omega>1/63.$}
% \vspace{5pt}
\begin{tabular}{|c|c|c|c|c|}\cline{1-5} 
\multicolumn{5}{|c|}{Implicit--explicit six-step  scheme} \\  \hline 
$\tau{\vphantom{\sum^{\sum^\sum}}}$ & $\|u(T)-u^N\|_{\mathcal{V}}$   &   Rate  \ \  & $\|p(T)-p^N\|_{\mathcal{H_W}}$     & Rate    \\*[2pt]
\hhline{|=====|}
  1/50&     1.1579e-07  &         &4.0664e-08  &             \\
  1/100&    5.8366e-09  &\ \ 4.3103   &5.6018e-10  &\ \ 6.1817      \\
  1/150&    1.8496e-08  &$-$2.8446   &1.0504e-09  &$-$1.5505        \\
  1/200&    1.1540e-07  &$-$6.3642   &6.6737e-09  &$-$6.4272          \\ \hline
\multicolumn{5}{|c|}{Fully implicit six-step scheme} \\ \hline
$\tau{\vphantom{\sum^{\sum^\sum}}}$ &        $\|u(T)-u^N\|_{\mathcal{V}}$    &   Rate  \ \  & $\|p(T)-p^N\|_{\mathcal{H_W}}$     & Rate    \\*[2pt]
\hhline{|=====|}
  1/50&     1.3567e-08  &         &3.5731e-08  &            \\
  1/100&    2.0384e-10  &6.0565   &5.3684e-10  &6.0566      \\
  1/150&    1.7666e-11  &6.0319   &4.6520e-11  &6.0321        \\
  1/200&    3.1537e-12  &5.9894   &8.2743e-12  &6.0022          \\ \hline
    \end{tabular}\label{table:2}
    }
  \end{center}
\end{table}
%
%{\color{blue}
A prominent advantage of these higher-order schemes is that, with almost the  computational cost 
of first-order schemes,  \cite{AMU:21}, they greatly improve the accuracy.
%The prominent advantage is that higher-order schemes can keep the same computation cost with first-order schemes \cite{AMU:21} but greatly improve the accuracy.
Tables \ref{table:1} and  \ref{table:2} show that  the fully implicit scheme \eqref{implicit} do not require any type of coupling condition. 
The implicit--explicit scheme \eqref{explicit} attain sixth-order accuracy under the coupling condition $\omega=0.015<1/63$; conversely, 
the scheme diverges when $\omega=0.03>1/63$, which is consistent with the results presented in Theorem  \ref{theorem:semi}. 
The sharpness of the weak coupling condition is further investigated in the following subsection.}
%}
\end{example}

\subsection{Sharpness of the weak coupling condition}
We proceed to present a numerical example aimed at verifying the requirement of the weak coupling condition stated in Lemma \ref{Le:nece-cond}. 
For this purpose, we consider the following test problem conforming to \eqref{weak}, where $\mathcal{V}=\mathcal{H_V}=\R^3, \mathcal{W}=\mathcal{H_W}=\R^1$, 
and the bilinear forms are specified as in \cite{AMU:24}:
\begin{equation*}
a(u,v)=v^{\top}Au,\quad d(v,p)=\sqrt{\omega}p^{\top}Dv,\quad
c(p,q)=q^{\top}Cp,\quad b(p,q)=q^{\top}Bp
\end{equation*}
with matrices
\begin{equation*}
A:=\frac{1}{2-\sqrt{2}}
  \begin{pmatrix*}[r]
    2 & -1 &    0 \\
     -1 &    2 & -1 \\
       0 & -1 &   2 
   \end{pmatrix*},\quad
   D:=\frac{1}{\sqrt{13(2-\sqrt{2})}}\begin{pmatrix}
   2 & 1 & 2
   \end{pmatrix},\quad
   C:=1,\quad B:=1.
\end{equation*}
The prefactor of $A$ is chosen such that the smallest eigenvalue $c_a$ of $A$ equals $1$. Additionally, with $c_c=1$ 
and the continuity constant of $d$ given by $C_d=\sqrt{\omega}$, \eqref{strength} is satisfied.
In this case, $DA^{-1}D^{\top}=1$ and $\M=\omega\C$.

We then examine our implicit--explicit scheme \eqref{explicit} using various time step sizes $\tau$ and varying coupling coefficients $\omega$.
The errors are evaluated at the final time $T=1$. As exact solution, we choose
\begin{equation*}
u(t)=\begin{pmatrix}
\sin(t)\\ \cos(t)\\ \e^t
\end{pmatrix},
\quad 
p(t)=(2t)^7+1.
\end{equation*}
The corresponding results are shown in Figure \ref{Fig:error}; we see that the critical value for stability is roughly $1/63\approx0.015873$, 
which satisfies the weak coupling condition stated in Lemma \ref{Le:nece-cond}, and demonstrates that the coupling condition is nearly sharp.
\begin{figure}[htb]  
\centering
  \begin{tabular}{cc}
      \includegraphics[width=8cm]{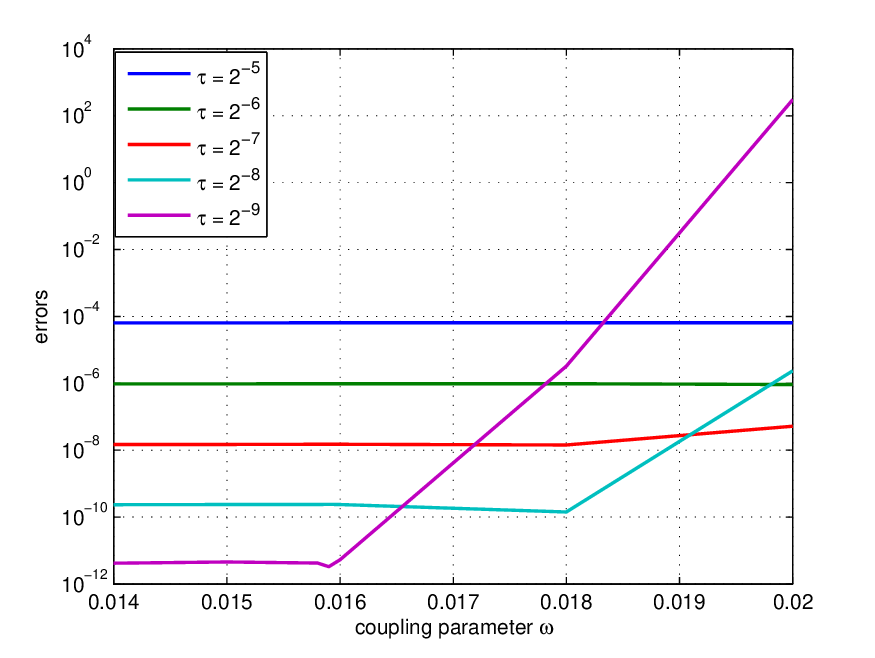}
  \end{tabular}
  \caption{Errors of the sixth-order implicit--explicit method for various
coupling parameters $\omega$ and various time step sizes $\tau$.}
  \label{Fig:error}
\end{figure}

%\section*{Acknowledgements}
%The work was supported by NSFC 12471381 and the Science Fund for Distinguished Young Scholars of Gansu Province under 
%Grant No.\ 23JRRA1020. Fan Yu is supported by NSFC 12501563 and the Postdoctoral Project of Hubei Province under Grant Number 2024HBBHJD064.

\bibliographystyle{amsplain}

\begin{thebibliography}{10}
%%%%%%%%%%%%%%%%%%%%%%%%%%%%%%%%%%%%%%%%%%%%%%

\bibitem{A1}
G. Akrivis,
\newblock \emph{Implicit--explicit multistep methods for nonlinear parabolic equations},
\newblock Math. Comp. \textbf{82} (2013) 45--68.
\href{https://doi.org/10.1090/S0025-5718-2012-02628-7}{DOI 10.1090/S0025-5718-2012-02628-7}.
\MR{2983015}


\bibitem{A2}
G. Akrivis,
\newblock \emph{Stability of implicit and implicit--explicit multistep methods for
nonlinear parabolic equations}, 
\newblock IMA J. Numer. Anal.  \textbf{38} (2018) 1768--1796.
\href{https://doi.org/10.1093/imanum/drx057}{DOI 10.1093/imanum/drx057}.
\MR{3867382}




\bibitem{ACYZ:21}
G. Akrivis,  M. H. Chen, F. Yu, and Z. Zhou,
\newblock \emph{The energy technique for the six-step BDF method},
\newblock SIAM J. Numer. Anal.  \textbf{59} (2021) 2449--2472.
\href{https://doi.org/10.1137/21M1392656}{DOI 10.1137/21M1392656}.
\MR{4316580}


\bibitem{AC}
G. Akrivis and M. Crouzeix,
\newblock \emph{Linearly implicit methods for nonlinear parabolic equations},
\newblock Math. Comp. \textbf{73} (2004) 613--635.
\href{https://doi.org/10.1090/S0025-5718-03-01573-4}{DOI 10.1090/S0025-5718-03-01573-4}.
\MR{2031397}



\bibitem{ACM2}
G. Akrivis, M. Crouzeix, and Ch. Makridakis,
\newblock \emph{Implicit--explicit multistep methods for quasilinear
parabolic equations},
\newblock Numer. Math. \textbf{82} (1999) 521--541.
\href{https://doi.org/10.1007/s002110050429}{DOI 10.1007/s002110050429}.
\MR{1701828}


\bibitem{AKK}
 G. Akrivis, O. Karakashian, and F. Karakatsani,
\newblock \emph{Linearly implicit methods for nonlinear evolution equations},
\newblock Numer. Math. \textbf{94} (2003) 403--418.
\href{https://doi.org/10.1007/s00211-002-0432-y}{DOI 10.1007/s00211-002-0432-y}.
\MR{1981162}





\bibitem{AK:16}
G. Akrivis and E. Katsoprinakis,
\newblock \emph{Backward difference formulae\emph{:} New multipliers and stability properties for parabolic equations},
\newblock  Math. Comp. \textbf{85} (2016) 2195--2216.
\href{https://doi.org/10.1090/mcom3055}{DOI 10.1090/mcom3055}.
\MR{3511279}

\bibitem{AMU:21}
R. Altmann, R. Maier, and B. Unger,
\newblock \emph{Semi-explicit discretization schemes for weakly-coupled
elliptic-parabolic problems},
\newblock  Math. Comp. \textbf{90} (2021) 1089--1118.
\href{https://doi.org/10.1090/mcom/3608}{DOI 10.1090/mcom/3608}.
\MR{4232218}

\bibitem{AMU:24}
R. Altmann, R. Maier, and B. Unger,
\newblock \emph{Semi-explicit integration of second order for weakly coupled
poroelasticity},
\newblock  BIT Numer. Math. \textbf{64} (2024) 20.
\href{https://doi.org/10.1007/s10543-024-01021-0}{DOI 10.1007/s10543-024-01021-0}.
\MR{4728769}

\bibitem{AMU:25}
R. Altmann, A. Mujahid, and B. Unger,
\newblock \emph{Decoupling multistep schemes for elliptic-parabolic problems},
\newblock arXiv.2407.18594v2
\href{https://doi.org/10.48550/arXiv.2407.18594}{arXiv.2407.18594}.

%{\color{red}Is this the paper:
%
%R. Altmann, A. Mujahid, and B. Unger,
%\newblock \emph{Higher-order iterative decoupling for poroelasticity},
%\newblock Adv. Comput. Math.  \textbf{50} (2024) 111.
%\href{https://doi.org/10.1007/s10444-024-10200-0}{DOI 10.1007/s10444-024-10200-0}.}


\bibitem{BC}
C. Baiocchi and M. Crouzeix,
\newblock \emph{On the equivalence of A-stability and G-stability},
\newblock Appl. Numer. Math. \textbf{5} (1989) 19--22.
\href{https://doi.org/10.1016/0168-9274(89)90020-2}{DOI 10.1016/0168-9274(89)90020-2}.
%Recent theoretical results in numerical ordinary differential equations.
\MR{979543}

\bibitem{Biot:41}
M. A. Biot, 
\newblock \emph{General theory of three-dimensional consolidation},
\newblock J. Appl. Phys. \textbf{12} (1941) 155--164.
\href{https://doi.org/10.1063/1.1712886}{DOI 10.1063/1.1712886}.

\bibitem{Brezis:11}
H. Brezis. 
\newblock \emph{Functional Analysis, Sobolev Spaces and Partial Differential Equations}.
\newblock Springer, New York, London, 2011.
\href{https://doi.org/10.1007/978-0-387-70914-7}{DOI 10.1007/978-0-387-70914-7}.
\MR{2759829}

\bibitem{Chan:07}
R. H. Chan and X. Q. Jin,
\newblock \emph{An Introduction to Iterative Toeplitz Solvers},
\newblock SIAM, Philadelphia, PA, 2007.
\href{https://doi.org/10.1137/1.9780898718850}{DOI 10.1137/1.9780898718850}.
\MR{2376196}



\bibitem{D}
G. Dahlquist,
\newblock \emph{G-stability is equivalent to A-stability},
\newblock BIT \textbf{18} (1978) 384--401.
\href{https://doi.org/10.1007/BF01932018}{DOI 10.1007/BF01932018}.
\MR{520750}

\bibitem{ERT:23}
E. Eliseussen, M. E. Rognes and T. B. Thompson,
\newblock \emph{A posteriori error estimation and adaptivity for multiple-network poroelasticity},
\newblock ESAIM Math. Model. Numer. Anal. \textbf{57} (2023) 1921--1952.
\href{https://doi.org/10.1051/m2an/2023033}{DOI 10.1051/m2an/2023033}.
\MR{4609868}

\bibitem{EM:09}
A. Ern and S. Meunier, 
\newblock \emph{A posteriori error analysis of Euler-Galerkin approximations to coupled elliptic-parabolic problems},
\newblock ESAIM Math. Model. Numer. Anal. \textbf{43} (2009) 353--375. 
\href{https://doi.org/10.1051/m2an:2008048}{DOI 10.1051/m2an:2008048}.
\MR{2512500}


%\bibitem{HNW}
%E. Hairer, S. P. Nørsett and G. Wanner,
%\newblock
%\emph{Solving Ordinary Differential Equations I: Nonstiff Problems}.
%\newblock Springer Series in Computational Mathematics v. 8,
%Springer--Verlag, Berlin, $2^\text{nd}$ revised ed., 1993,
%corr. $2^\text{nd}$ printing, 2000.
%\href{https://doi.org/10.1007/978-3-540-78862-1}{DOI  10.1007/978-3-540-78862-1}.
%\MR{1227985}
%


\bibitem{HW}
E. Hairer and G. Wanner,
\newblock
\emph{Solving Ordinary Differential Equations II\emph{:} Stiff and
Differential--Algebraic Problems},
\newblock
2$^\text{nd}$ revised ed., Springer--Verlag, Berlin Heidelberg,
Springer Series in Computational Mathematics v.\ 14, 2010.
\href{https://doi.org/doi:10.1007/978-3-642-05221-7}{DOI  10.1007/978-3-642-05221-7}.
\MR{2657217}


\bibitem{LMV}
C. Lubich, D. Mansour, and C. Venkataraman,
\newblock \emph{Backward difference time discretization of parabolic differential equations
on evolving surfaces},
\newblock IMA J. Numer. Anal. \textbf{33} (2013) 1365--1385.
\href{https://doi.org/10.1093/imanum/drs044}{DOI 10.1093/imanum/drs044}.
\MR{3119720}



\bibitem{NO}
O. Nevanlinna and F. Odeh,
\newblock \emph{Multiplier techniques for linear multistep methods},
\newblock Numer. Funct. Anal. Optim. \textbf{3} (1981) 377--423.
\href{https://doi.org/10.1080/01630568108816097}{DOI 10.1080/01630568108816097}.
\MR{636736}

\bibitem{STW:2011}
J. Shen, T. Tang, and L. Wang,
\newblock \emph{Spectral Methods: Algorithms, Analysis and Applications},
 \newblock Springer--Verlag, Berlin, 2011.
 \href{https://doi.org/10.1007/978-3-540-71041-7}{DOI 10.1007/978-3-540-71041-7}.
 \MR{1311481}

\bibitem{Sh:00}
R. E. Showalter,
\newblock \emph{Diffusion in poro-elastic media},
\newblock J. Math. Anal. Appl. \textbf{251} (2000) 310--340.
\href{https://doi.org/10.1006/jmaa.2000.7048}{DOI 10.1006/jmaa.2000.7048}.
\MR{1790411}

\end{thebibliography}

\end{document}